\documentclass[12pt,oneside,english,draft]{amsart}
\usepackage[T1]{fontenc}
\usepackage[latin9]{inputenc}
\usepackage{geometry}
\usepackage{mathrsfs}
\usepackage{amsthm}
\usepackage{amstext}
\usepackage{amssymb}
\usepackage{esint}
\usepackage{color}

\makeatletter
\numberwithin{equation}{section}
\numberwithin{figure}{section}
\theoremstyle{plain}
\newtheorem{thm}{\protect\theoremname}[section]
  \theoremstyle{definition}
  
  \theoremstyle{plain}
  \newtheorem{prop}[thm]{\protect\propositionname}
  \theoremstyle{plain}
  \newtheorem{cor}[thm]{\protect\corollaryname}
  \theoremstyle{plain}
  \newtheorem{lem}[thm]{\protect\lemmaname}
   \theoremstyle{plain}
  \newtheorem{hypo}[thm]{\protect\hypothesisname}
     \theoremstyle{plain}
  \newtheorem{rem}[thm]{Remark}
  
\def\ra{{\rightarrow}}
\makeatother

\usepackage{babel}
  \providecommand{\corollaryname}{Corollary}
  \providecommand{\definitionname}{Definition}
  \providecommand{\lemmaname}{Lemma}
  \providecommand{\propositionname}{Proposition}
\providecommand{\theoremname}{Theorem}
\providecommand{\hypothesisname}{Hypothesis}
\def\eps{{\epsilon}}
\newcommand{\R}{\mathbb{R}}

\newcommand{\YY}{{\mbox{{\bf Y}}}}
\newcommand{\yy}{{\mbox{{\bf y}}}}
\newcommand{\zz}{{\mbox{{\bf z}}}}
\newcommand{\xxi}{{\mbox{\boldmath$\xi$}}}
\newcommand{\eeta}{{\mbox{\boldmath$\eta$}}}

\begin{document}

\title{Transport maps for $\beta$-matrix models and Universality}

\author{F. Bekerman$^\ast$, A. Figalli$^\dagger$, A. Guionnet$^{\ddagger}$ }

\thanks{$^\ast$ Department of Mathematics, Massachusetts Institute of Technology, 77 Massachusetts Ave, Cambridge, MA 02139-4307 USA. email:
\texttt{bekerman@mit.edu}.\\
$^\dagger$ The University of Texas at Austin,
Mathematics Dept. RLM 8.100,
2515 Speedway Stop C1200,
Austin, Texas 78712-1202, USA. email: \texttt{figalli@math.utexas.edu}.\\
$^{\ddagger}$ Department of Mathematics, Massachusetts Institute of Technology, 77 Massachusetts Ave, Cambridge, MA 02139-4307 USA. email:
\texttt{guionnet@math.mit.edu}.}

\begin{abstract}
We construct approximate transport maps  for non-critical $\beta$-matrix models, 
that is, maps so that the push forward of a non-critical $\beta$-matrix model with a given potential
is a non-critical $\beta$-matrix model with another potential, up to a small error in the total variation distance. One of the main features of our construction is that these maps enjoy regularity estimates
which are uniform in the dimension. In addition, we find a very useful asymptotic expansion for such maps which allow us to deduce that local statistics have the same asymptotic 
behavior for both models. 
\end{abstract}
\maketitle

\section{Introduction.}
Given a potential $V:\R\to \R$ and $\beta>0$,
we consider  the $\beta$-matrix model 
\begin{equation}\label{betamodel}\mathbb P^N_V(d\lambda_1,\ldots,d\lambda_N):=\frac{1}{Z^N_V}\prod_{1\le i<j\le N}|\lambda_i-\lambda_j|^\beta e^{-N\sum_{i=1}^N V(\lambda_i)} \,d\lambda_1\dots\,d\lambda_N,\end{equation}
where $Z_V^N:=\int \prod_{1\le i<j\le N}|\lambda_i-\lambda_j|^\beta e^{-N\sum_{i=1}^N V(\lambda_i)} \,d\lambda_1\dots\,d\lambda_N$.
  It is well known (see e.g. \cite{Vi}) that  for any $V,W:\R\to\R$ such that $Z_V^N,Z_{V+W}^N<\infty$ there exists  maps $T^N:\mathbb R^N\ra\mathbb R^N$  that transport  $\mathbb P^N_V$ into
$\mathbb P^N_{V+W}$,  that is, for any bounded measurable function $f:\mathbb R^N\ra\mathbb R$ we have
$$\int f\circ T^N (\lambda_1,\ldots,\lambda_N) \,\mathbb P^N_V(d\lambda_1,\ldots,d\lambda_N)=\int f(\lambda_1,\ldots,\lambda_N) \,  \mathbb P^N_{V+W}(d\lambda_1,\ldots,d\lambda_N)\,.$$
However, the dependency in the dimension $N$ of this transport map is in general unclear
unless one makes very strong  assumptions on the densities \cite{Caff} that unfortunately are never satisfied in our situation.
Hence, it seems extremely difficult to use these maps $T^N$ to understand
the relation between  the asymptotic  of the two models.

The main contribution of this paper is to show that
a variant of this approach is indeed possible and provides a very robust and flexible method to compare the asymptotics of local statistics.
In the more general context of several-matrices models, it was shown in \cite{GS} that the maps $T^N$ are asymptotically well 
 approximated by a function of matrices independent of $N$, but it was left open the question of studying  
 corrections to this limit. In this article we consider one-matrix models, and more precisely their generalization given by 
$\beta$-models, and we 
construct {\it approximate} transport maps with a very precise 
dependence on the dimension. This allows us to compare local fluctuations of the eigenvalues and show universality.\\

Local fluctuations were first studied in the case $\beta=2$. After the work of Mehta \cite{ME} where the sine kernel law was exhibited for the GUE,
Pastur and Shcherbina \cite{PS} proved universality for $V\in C^3$. The Riemann Hilbert approach was then developed in this context in \cite{DKMcLVZ2} for analytic potentials.
The paper \cite{LL} provides the most general result for universality in the case $\beta=2$. 

In the case $\beta=1,4$, universality was proved in  \cite{DeGibulk} in the bulk, and \cite{DeGiedge} at the edge, for monomials $V$ (see \cite{DeGi2} for a review). For general one-cut potentials, the first proof of universality was given in \cite{Shch} in the case $\beta=1$, whereas \cite{KrSh} treated the case $\beta=4$ in the one-cut case.
The multi-cut case was then considered in \cite{Sh}.

The local fluctuations of more general $\beta$-ensembles were only derived recently \cite{VV,RRV} in the Gaussian case. 
Universality in the $\beta$-ensembles was first addressed in \cite{BEY1} (in the bulk, $\beta> 0$, $V\in C^4$), then in  \cite{BEY2} (at the edge, $\beta\ge 1$, $V\in C^4$) and 
\cite{KRV} (at the edge, $\beta>0$, $V$ convex polynomial) and finally in  \cite{Shch} (in the bulk, $\beta>0$, $V$ analytic, multi-cut case included). 

The paper \cite{Shch} and the current one were completed essentially at the same time,
and the method developed in \cite{Shch} is closely related with the one followed in the present work in the sense that it uses a transport map, but still they are very different.
Indeed, \cite{Shch} uses a global transport map $T^N=T_0^{\otimes N}$, independent of the dimension,  which maps the probability $\mathbb P^N_{V}$ to an intermediate probability measure $\tilde {\mathbb P}^N_{V+W}$ 
which is absolutely continuous with respect to $\mathbb P^N_{V+W}$ but with a non-trivial density, and then shows that this density does not perturb the local fluctuations of local statistics.
On the other hand, in the present paper the map $T^N$ is constructed to map $\mathbb P^N_{V}$ to an intermediate probability measure $\tilde {\mathbb P}^N_{V+W}$ 
with density with respect to $\mathbb P^N_{V+W}$ going uniformly to one: the dependence of $T_N$ in the dimension $N$ then allows  us to retrieve the universality of the local fluctuations. 
\\

The main idea of this paper is that we can build
a smooth approximate transport map which at the first order is simply a tensor product, and then it has
first order corrections of order $N^{-1}$ (see Theorem \ref{thm:T}). This continues the long standing idea developed in  loop equations theory to study the correlation functions of
$\beta$-matrix models by clever change of variables, see e.g. \cite{johansson, KrSh, AM, BEMP}.
On the contrary to loop equations, the change of variable has to be of order one rather than infinitesimal. 

The first order term of our map is simply the monotone transport map
between the  asymptotic equilibrium measures; then, corrections to this first order are constructed so that densities match up to a priori 
smaller fluctuating terms.
As we shall see, our transport map is constructed as the flow
of a vector field obtained by approximately solving a linearized Jacobian equation
and then making a suitable ansatz on the structure of the vector field. The errors are
 controlled by deriving bounds on  covariances and correlations functions thanks to loop equations,  allowing us to obtain a self-contained
proof of universality (see Theorem \ref{thm:univ}). Although optimal transportation will never be really used,
it will provide us the correct intuition in order to solve this problem (see Sections \ref{sect:MA2} and \ref{sect:approx}).

We notice that this last step could also be performed either by using directly central limit theorems, see e.g.  \cite{Shch,borot-guionnet,Shc12},  or local limit laws
\cite{BEY1,BEY2}.  However, our approach has the advantage of being pretty robust and should generalize to many other mean field
models. In particular, it should be possible to generalize it to the several-matrices models, at least in the perturbative regime considered in 
\cite{GS}, but this would not solve yet the question of universality  as the transport map would be a non-commutative function of several matrices.
Even in the case of GUE matrices, there are not yet any results about the local statistics of the eigenvalues of such non-commutative functions, except for a few very specific cases. 
\\ 

We now describe our results in detail.
Given a potential $V:\mathbb R\to \mathbb R$, we consider the probability measure \eqref{betamodel}. We assume that $V$ goes to infinity faster than $\beta\log |x|$  (that is $V(x)/\beta\log |x|\to +\infty$ as $|x|\to +\infty$) so that in particular $Z^N_V$ is finite.

We will use
$\mu_V$ to denote the equilibrium measure, which is obtained as limit of
the spectral measure and is characterized as the unique minimizer (among probability measures) 
of \begin{equation}\label{entr}I_V(\mu):=\frac{1}{2} \int \bigl( V(x)+V(y) -\beta \log |x-y|\bigr)d\mu(x)d\mu(y)\,.\end{equation}
We assume hereafter that another smooth potential $W$ is given so that $V+W$ goes to infinity faster than $\beta\log |x|$.  We set $V_t:=V+tW$, and
we shall make the following assumption:
\begin{hypo}\label{hypo1}
We assume that $\mu_{V_0}$ and $\mu_{V_1}$ have a connected support and  are non-critical, that is, there exists a constant $\bar c>0$ such that, for $t=0,1$,
$$\frac{d\mu_{V_t}}{dx}=S_t(x)\sqrt{ (x-a_t)(b_t-x)} \qquad \text{with $S_t \geq \bar c$ a.e. on  $[a_t,b_t]$.}$$
\end{hypo}

\begin{rem}{\rm 
The assumption of a connected support could be removed here, following the lines of \cite{Shch,borot-guionnet2}. Indeed, only a generalization of Lemma \ref{central} is required, which is not difficult. However, the non-criticality assumption
cannot be removed easily, as criticality would result in singularities in the transport map.}
\end{rem}

Finally, we assume that the eigenvalues stay  in a neighborhood of the support $[a_t-\eps,b_t+\eps]$ with large enough  $\mathbb P^N_V$-probability, that is with probability greater than $1-C\,N^{-p}$ for some $p$ large enough. 
By \cite[Lemma 3.1]{borot-guionnet2},
 the latter is fulfilled as soon as:
 \begin{hypo} \label{hypo2} For $t=0,1$,
 \begin{equation}\label{noncr}
 U_{V_t}(x):=V_t(x)-\beta\int d\mu_{V_t}(y)\log |x-y|\end{equation}
 achieves its minimal value on $[a,b]^c$ at its boundary $\{a,b\}$ 
 \end{hypo}
All these assumptions are verified for instance if $V_t$ is uniformly convex for  $t=0,1$. \\

The main goal of this article is to build an approximate transport map between $\mathbb P^N_V$ and $\mathbb P^N_{V+W}$: more precisely, we construct a map $T^N:\mathbb R^N\ra\mathbb R^N$ 
such that, for any bounded measurable function $\chi$,
\begin{equation}\label{aptr}\left|\int \chi\circ T^N\, d\mathbb P^N_{V}-\int \chi\, d\mathbb P^N_{V+W}\right|\le
C\,\frac{(\log N)^3 }{{N}} \|\chi\|_\infty
\end{equation}
for some constant $C$ independent of $N$, and which has a very precise
expansion in the dimension (in the following result, $T_0:\R\to\R$ is a  smooth  transport map  of $\mu_{V}$ onto $\mu_{V+W}$, see Section \ref{sect:flow}):

\begin{thm}
\label{thm:T}
Assume that $V',W$ are of class $C^{30}$ and satisfy 
Hypotheses \ref{hypo1} and \ref{hypo2}.
Then there exists a map $T^N=(T^{N,1},\ldots,T^{N,N}):\mathbb R^N\ra\mathbb R^N$ which satisfies \eqref{aptr} and has the form
$$T^{N,i}(\hat\lambda)= T_0(\lambda_i) +\frac{1}{N} T_1^{N,i}(\hat\lambda)\quad \forall\,i=1,\ldots,N,\qquad \hat\lambda:=(\lambda_1,\ldots,\lambda_N),$$
where $T_0:\R\to \R$ and $T_1^{N,i}:\R^N\to \R$ are smooth and satisfy uniform (in $N$) regularity estimates. More precisely, $T^N$ is of class 
 $C^{23}$ and 
we have the decomposition  $T^{N,i}_1= X^{N,i}_1+\frac{1}{N} X^{N,i}_2$ where
\begin{equation}\label{boi1int}
\sup_{1\leq k \leq N}\|X_{1}^{N,k}\|_{L^4(\mathbb P_{V}^N)} \leq C\log N,\qquad \|X_{2}^{N}\|_{L^2(\mathbb P_{V}^N)} \leq CN^{1/2}(\log N)^2,
\end{equation}for some constant $C>0$ independent of $N$.
In addition, with probability greater than $1- N^{-N/C}$,
\begin{equation}\label{boi2int}\max_{1\le k,k'\le N}
|X_{1}^{N,k}(\lambda)-X_{1}^{N,k'}(\lambda)|\le C\,\log N \sqrt{N}|\lambda_k-\lambda_{k'}|.\end{equation}
\end{thm}

As we shall see in Section \ref{sect:univ}, this result implies universality as follows
(compare with \cite[Theorem 2.4]{BEY2}):

\begin{thm} 
\label{thm:univ}
Assume
$V',W\in C^{30}$, and let $T_0$ be as in Theorem \ref{thm:T} above.
Denote $\tilde P_V^N$ the distribution of
the increasingly ordered  eigenvalues $\lambda_i$ under $\mathbb P_V^N$. There exists a constant $\hat C>0$, independent of $N$, such that the following two facts hold true:
\begin{enumerate}
\item Let $M\in (0,\infty)$ and $m\in \mathbb N$. For any Lipschitz function $f:\R^m\to \R$ supported inside  $[-M,M]^m$,
\begin{multline*}
\bigg|\int f\bigl(N(\lambda_{i+1}-\lambda_i),\ldots,N(\lambda_{i+m}-\lambda_i)\bigr)\, d\tilde P_{V+W}^N\\
\qquad -\int f\bigl(T_{0}'(\lambda_i)N(\lambda_{i+1}-\lambda_i),\ldots,T_{0}'(\lambda_i)N(\lambda_{i+m}-\lambda_i)\bigr) d\tilde P_{V}^N\bigg|\\
 \le \hat C \,\frac{(\log N)^3 }{N} 
\|f\|_\infty +  \hat C\,\biggl(\sqrt{m}\,\frac{(\log N)^2}{N^{1/2}}+M\,\frac{\log N}{N^{1/2}} +\frac{M^2}{N}\biggr)
\|\nabla f\|_{\infty}.
\end{multline*}
\item 
Let $a_V$ (resp. $a_{V+W}$) denote the smallest point in the
support of $\mu_V$ (resp. $\mu_{V+W}$), so that ${\rm supp}(\mu_V)\subset [a_V,\infty)$
(resp.  ${\rm supp}(\mu_{V+W})\subset [a_{V+W},\infty)$).
Let $M\in  (0,\infty)$. Then, for any Lipschitz function $f:\R^m\to \R$ supported inside  $[-M,M]^m$,
\begin{multline*}
\bigg|
\int f\bigl(N^{2/3}(\lambda_1-a_{V+W}), \ldots,N^{2/3}(\lambda_m-a_{V+W})\bigr) \,d\tilde P_{V+W}^N\\
-\int f\Bigl(N^{2/3}T_{0}'(a_V)\bigl(\lambda_1-a_V\bigr), \ldots,N^{2/3}T_{0}'(a_V)\bigl(\lambda_m-a_V\bigr)\Bigr)\, d\tilde P_{V}^N\bigg|\\
\le  \hat C \,\frac{(\log N)^3}{{N}} 
\|f\|_\infty+ \hat C\biggl(\sqrt{m}\,\frac{(\log N)^2}{N^{5/6}}+M\,\frac{\log N}{N^{5/6}} +\frac{M^2}{N^{4/3}} + \frac{\log N}{N^{1/3}}\biggr)
\|\nabla f\|_{\infty} .
\end{multline*}
The same bound holds around the largest point in the support of $\mu_V$.
\end{enumerate}
\end{thm}

\begin{rem} {\em The condition that $V',W\in C^{30}$ in the theorem above is clearly non-optimal
(compare with \cite{BEY1}). 
For instance, by using Stieltjes transform instead of Fourier transform in some of our estimates,
we could reduce the regularity assumptions on $V',W$ to $C^{21}$ by a slightly more cumbersome proof.
In addition, by using \cite[Theorem 2.4]{BEY2}
we could also weaken our regularity assumptions in Theorem \ref{thm:T},
as we could use that result to estimate the error terms in Section \ref{sect:error}.
However, the main point of this hypothesis for us is to stress that we do not need to have 
analytic potentials, as often required in matrix models theory. Moreover, under this assumption we  can provide self-contained and short proofs of Theorems \ref{thm:T} and \ref{thm:univ}. 
}
\end{rem} 
Note that Theorem \ref{thm:T} is well suited to prove universality of the spacings distribution  in the bulk as stated in Theorem \ref{thm:univ}, but
it is not clear how to directly deduce the universality of the rescaled density, see e.g \cite[Theorem 2.5(i)]{BEY1}. Indeed, this corresponds 
to choosing test functions whose uniform norm blows up like some power of the dimension, so to apply Theorem \ref{thm:T} we should have an a priori  control 
on the numbers of eigenvalues inside sets of size of order $N^{-1}$
under both $\mathbb P^N_{V}$ and $\mathbb P^N_{V+W}$. Notice however that \cite[Theorem 2.5(ii)]{BEY1} requires $\beta \geq 1$,
while our results hold for any $\beta>0$. \cite{KRV} holds for $\beta>0$ but convex polynomial.
In particular, the edge universality proved in Theorem \ref{thm:univ}(2) is  new for $\beta \in (0,1)$ and $V \in C^{30}$ which is not
a polynomial.
In addition our strategy is very robust and flexible.
For instance, although we shall not pursue this direction here, it looks likely to us that one could use it prove the universality of the  asymptotics of the law of $\{N(\lambda_i-x) \}_{1\le i\le m}$ under $\tilde P_V^N$  for
given $i$ and $x$, as in \cite{BY14,BY14B}. Indeed, $N(\lambda_i-x)$ can be seen via the approximate transport map to be given by 
the same kind of difference under the source measure (e.g. the Gaussian one), up to a term coming from the transport map $T^{N,i}_1$. Proving that $T^{N,i}_1$ does not depend on $i$ locally would suffice to prove universality, as the sine-kernel law is translation invariant. 
\\

The paper is structured as follows:
In Section \ref{sect:MA} we describe the general strategy to construct our transport map as the flow
of vector fields obtained by approximately solving a linearization of the Monge-Amp\`ere equation (see \eqref{eq:laplace}).
As we shall explain there, this idea comes from optimal transport theory.
In Section \ref{sect:laplace} we make an ansatz on the structure of an approximate solution to \eqref{eq:laplace} and we show that our ansatz actually provides a smooth solution which enjoys
very nice regularity estimates that are uniform as $N \to \infty$.
In Section \ref{sect:flow} we reconstruct the approximate transport map from $\mathbb P^N_V$ to $\mathbb P^N_{V+W}$ via a flow argument.
The estimates proved in this section will be crucial in Section \ref{sect:univ} to show universality.\\

\textit{Acknowledgments:} 
AF was partially supported by NSF Grant DMS-1262411. AG was partially supported by the Simons Foundation and by NSF Grant DMS-1307704. 
Both authors are thankful to the anonymous referees for several useful comments that helped to improve the presentation of the paper.

\section{Approximate Monge-Amp\`ere equation}
\label{sect:MA}
\subsection{Propagating the hypotheses}
The central idea of the paper  is to build transport maps as flows, and in fact to build transport maps between $\mathbb P_V^N$
and $\mathbb P_{V_t}^N$ where $t\mapsto V_t$ is a smooth function so that $V_0=V$, $V_1=V+W$. 
In order to have a good interpolation between $V$ and $V+W$, it will be convenient to assume that
the support of the two equilibrium measures $\mu_V$ and $\mu_{V+W}$ (see \eqref{entr}) are the same. 
This can always be done up to an affine transformation. Indeed, if $L:\R\to \R$ is the affine transformation
which maps $[a_1,b_1]$ (the support of $\mu_{V_1}$) onto $[a_0,b_0]$ (the support of $\mu_{V_0}$), we first construct a transport map from
$ \mathbb P_{V}^N$ to $L^{\otimes N}_\sharp \mathbb P_{V+W}^N=\mathbb P_N^{V+\tilde W}$
where \begin{equation}\label{pol}
\tilde W=V\circ L^{-1} +W\circ L^{-1} -V,\end{equation}
and then we simply compose our transport map with $(L^{-1})^{\otimes N}$ to get the desired map from $ \mathbb P_{V}^N$ to $\mathbb P_{V+W}^N$.
Hence, without loss of generality
we will hereafter assume that $\mu_V$ and $\mu_{V+W}$ have the same support.  We then consider the interpolation $\mu_{V_t}$ with $V_t=V+tW$,
$t\in [0,1]$. 
We have:

\begin{lem}\label{propa} If  Hypotheses \ref{hypo1} and \ref{hypo2} are fulfilled for $t=0,1$,  then Hypothesis \ref{hypo1}
is also fulfilled for all $t\in [0,1]$. Moreover, we may assume without loss of generality that $V$ goes to infinity as fast as we want 
up to modify $\mathbb P_{V}^N$ and $\mathbb P_{V+W}^N$ by a negligible error (in total variation).
\end{lem}
\begin{proof}
Let $\Sigma$ denote the support of $\mu_V$ and $\mu_{V+W}$.
Following \cite[Lemma 5.1]{borot-guionnet},  the measure $\mu_{V_t}$ is simply given by
$$\mu_{V_t}=(1-t)\mu_V+t \mu_{V+W}.$$
Indeed, $\mu_V$ is uniquely determined by the fact that there exists a constant $c$ such that
$$\beta\int\log |x-y|d\mu_V(x)-V\le c$$
with equality on the support of $\mu_V$, and this property extends to  linear combinations. 
As a consequence the support  of $\mu_{V_t}$  is $\Sigma$,
and its density is bounded away from zero on $\Sigma$. This shows that Hypothesis \ref{hypo1}
is fulfilled for all $t\in [0,1]$.

Furthermore, we can modify $\mathbb P_{V}^N$ and $\mathbb P_{V+W}^N$ outside an open neighborhood of $\Sigma$ without
changing the final result, as eigenvalues will quit  this neighborhood only
with very small probability under our assumption of non-criticality according to the large deviation 
estimates developped in  \cite{BAGD,borot-guionnet} and culminating in   \cite{borot-guionnet2} as follows:

\begin{align*}
\limsup_{N\ra\infty}\frac{1}{N}\ln \mathbb P^{V}_{N}\left[\exists\, i\,:\,\lambda_i \in F\right] & \le  -\frac{\beta}{2}\,\inf_{x \in F} \tilde{U}_{V}(x), \\
\liminf_{N\ra\infty}\frac{1}{N}\ln \mathbb P^{V}_{N}\left[\exists\, i\,:\,\lambda_i \in \Omega\right] & \ge  -\frac{\beta}{2}\,\inf_{x \in \Omega} \tilde{U}_{V}(x).
\end{align*}
where $\tilde{U}_V:=U_V-\inf U_V$, and $U_V$ is defined as in \eqref{hypo2}.
\end{proof}

Thanks to the above lemma and the discussion immediately before it, we can assume that $\mu_{V}$ and $\mu_{V+W}$ have the same support, that
$W$ is bounded, and that $V$ goes to infinity faster than $x^{p}$ for some $p>0$ large enough.

\subsection{Monge-Amp\`ere equation}
\label{sect:MA2}
Given the two probability densities $\mathbb P^N_{V_t}$ to $\mathbb P^N_{V_s}$
as in \eqref{betamodel} with $0 \leq t \leq s \leq 1$,
by optimal transport theory it is well-known that there exists a (convex) function $\phi_{t,s}^N$ such that
$\nabla \phi_{t,s}^N$ pushes forward $\mathbb P^N_{V_t}$ onto $\mathbb P^N_{V_s}$ and which satisfies
the Monge-Amp\`ere equation
$$
\det(D^2 \phi_{t,s}^N)=\frac{\rho_t}{\rho_s(\nabla \phi_{t,s}^N)},\qquad \rho_\tau:=\frac{d\mathbb P_{V_\tau}^N}{d\lambda_1\ldots\,d\lambda_N}
$$
(see for instance \cite[Chapters 3 and 4]{Vi} or the recent survey paper \cite{DF} for an account on optimal transport theory and its link
to the Monge-Amp\`ere equation).

Because $\phi_{t,t}(x)=|x|^2/2$ (since $\nabla\phi_{t,t}$ is the identity map),
we can differentiate the above equation with respect to $s$ and set $s=t$ to get
\begin{equation}
\label{eq:laplace}\Delta \psi^N_t= c_t^N-\beta\sum_{i<j}\frac{\partial_i \psi^N_t-\partial_j\psi^N_t}{\lambda_i-\lambda_j}
+N\sum_i W(\lambda_i) +N \sum_iV_t'(\lambda_i)\partial_i \psi^N_t,
\end{equation}
where $\psi_t^N:=\partial_s\phi_{t,s}^N|_{s=t}$ and
$$c^N_t:=-N\int \sum_iW(\lambda_i)  \,d\mathbb P^N_{V_t}= \partial_t \log Z^N_{V_t}\,.$$
Although this is a formal argument, it suggests to us a way to construct maps $T_{0,t}^N:\R^N\to\R^N$
sending $\mathbb P_{V}^N$ onto $\mathbb P_{V_t}^N$:
indeed, if $T_{0,t}^N$ sends $\mathbb P_{V}^N$ onto $\mathbb P_{V_t}^N$
then $\nabla \phi_{t,s}^N\circ T_{0,t}^N$ sends $\mathbb P_{V}^N$ onto $\mathbb P_{V_s}^N$.
Hence, we may try to find $T_{0,s}^N$ of the form $T_{0,s}^N=\nabla \phi_{t,s}^N\circ T_{0,t}^N+o(s-t)$.
By differentiating this relation 
with respect to $s$ and setting $s=t$ we obtain
$\partial_t T^N_{0,t}=\nabla  \psi^N_t(T^N_{0,t})$.

Thus, to construct a transport map $T^N$ from 
$\mathbb P^N_{V}$ onto $\mathbb P^N_{V+W}$ 
we could first try to find $\psi_t^N$ by solving \eqref{eq:laplace}, and then
construct  $T^N$ solving the ODE
$\dot X^N_t=\nabla \psi^N_t(X^N_t)$ and setting $T^N:=X_1^N$.
We notice that, in general, $T^N$ is not an optimal transport map for the quadratic cost.

Unfortunately, finding an exact solution of \eqref{eq:laplace} enjoying ``nice'' regularity estimates that 
are uniform in $N$ seems extremely difficult. So, instead,
we make  
an ansatz on the structure of $\psi_t^N$ (see \eqref{eq:psiN} below):
the idea is that at first order eigenvalues do not interact, then at order $1/N$ eigenvalues interact at most by pairs,
and so on.
As we shall see, in order to construct a function which enjoys nice regularity estimates  and satisfies \eqref{eq:laplace} 
up to a error that goes to zero as $N\to \infty$, it will be enough to stop the expansion at $1/N$.
Actually, while the argument before provides us the right intuition, 
we notice that there is no need to assume that the vector field generating the flow $X_t^N$ is a gradient,
so we will consider general vector fields $\YY_t^N=(\YY_{1,t}^N,\ldots,\YY_{N,t}^N):\R^N\to\R^N$ that approximately solve
\begin{equation}
\label{eq:laplace2}
{\rm div} \YY^N_t= c_t^N-\beta\sum_{i<j}\frac{\YY^N_{i,t}-\YY^N_{j,t}}{\lambda_i-\lambda_j}
+N\sum_i W(\lambda_i) +N \sum_i V_t'(\lambda_i) \YY^N_{i,t},
\end{equation}

We begin by checking that the flow of an approximate solution of  \eqref{eq:laplace2} gives an approximate transport map.

\subsection{Approximate Jacobian equation}
\label{sect:approx}
Here we show that if a $C^1$ vector field $\YY^N_t$ approximately satisfies \eqref{eq:laplace2}, 
then its flow
$$\dot X^N_t=\YY^N_t(X^N_t),\qquad X^N_0=\operatorname{Id},$$
produces almost a transport map.

More precisely, let $\YY^N_t:\R^N\to\R^N$ be a smooth vector field  and denote
$${\mathcal R}^N_t(\YY^N):= c_t^N-\beta\sum_{i<j}\frac{\YY^N_{i,t}-\YY^N_{j,t}}{\lambda_i-\lambda_j}
+N\sum_i W(\lambda_i) +N \sum_i V_t'(\lambda_i) \YY^N_{i,t}-{\rm div} \YY^N_{t}.$$

\begin{lem}
\label{lemma:flow}
Let $\chi:\mathbb R^N\to \mathbb R$ be a bounded measurable  function, and let $X^N_t$  be the flow of
$\YY^N_{t}$.
Then
$$\left|\int \chi(X^N_t)\, d\mathbb P^N_V- \int \chi\, d\mathbb P^N_{V_t}\right|\le\|\chi\|_\infty\int_0^t \|{\mathcal R}^N_s(\YY^N)\|_{L^1(\mathbb P^N_{V_s})}ds.$$
\end{lem}
\begin{proof}
Since $\YY^N_t\in C^1$, by Cauchy-Lipschitz Theorem its flow is a bi-Lipchitz homeomorphism.

If $JX^N_t$ denotes the Jacobian of $X^N_t$ and $\rho_t$ the density of $\mathbb P^N_{V_t}$, by the change of variable formula it follows
 that
 $$\int \chi\, d\mathbb P^N_{V_t}= \int \chi(X^N_t)JX^N_t \rho_t(X^N_t) \,dx$$
thus
\begin{equation}\label{coco}
\left|\int \chi(X^N_t) \,d\mathbb P^N_V- \int \chi\, d\mathbb P^N_{V_t}\right|\le\|\chi\|_\infty \int |\rho_0-JX^N_t \rho_t(X^N_t)| \,dx=:\|\chi\|_\infty A_t\end{equation}
Using that $\partial_t(JX^N_t )={\rm div} \YY^N_{t}\, JX^N_t$ and that the derivative of the norm is smaller than the norm of the derivative, we get
\begin{align*}
|\partial_t A_t|&\le\int \Bigl|\partial_t \bigl(JX^N_t \rho_t(X^N_t)\bigr)\Bigr| \,dx\\
&= \int|{\rm div} \YY^N_{t}\, JX^N_t\, \rho_t(X^N_t)+JX^N_t \,(\partial_t\rho_t)(X^N_t)+JX^N_t\,\nabla \rho_t(X^N_t)\cdot \partial_t X^N_t|\,dx\\
&=\int |{\mathcal R}^N_t(\YY)|(X^N_t)\,JX^N_t \,\rho_t(X^N_t)\,dx\\
&=\int |{\mathcal R}^N_t(\YY)|\,d\mathbb P^N_{V_t}.
\end{align*}
Integrating the above estimate in time completes the proof.
\end{proof}
By taking the supremum over all functions $\chi$ with $\|\chi\|_\infty \leq 1$, the lemma above gives:

\begin{cor}\label{cortr}
Let $X^N_t$  be the flow of
$\YY_t^N$, and set $\hat {\mathbb P}_t^N:=(X^N_t)_\#\mathbb P^N_V$ the image of $\mathbb P^N_V$ by $X^N_t$. Then
$$\|\hat{\mathbb P}_t^N - \mathbb P^N_{V_t}\|_{TV}\leq \int_0^t \|{\mathcal R}^N_s(\YY^N)\|_{L^1(\mathbb P^N_{V_s})}ds.$$
\end{cor}

\section{Constructing an approximate solution to \eqref{eq:laplace}}
\label{sect:laplace}

Fix $t \in [0,1]$ and define the random measures
\begin{equation}
\label{defemp}
L_N:=\frac{1}{N}\sum_{i}\delta_{\lambda_i}\quad\mbox{ 
and }\quad M_N:=\sum_i\delta_{\lambda_i}-N\mu_{V_t}.\end{equation}
As we explained in the previous section, a natural ansatz to find an approximate solution of 
\eqref{eq:laplace} is given by
\begin{equation}
\label{eq:psiN}
\psi^N_t(\lambda_1,\ldots,\lambda_N):=\int  \Bigl[\psi_{0,t}(x) +\frac{1}{N}  \psi_{1,t} (x) \Bigr]\,
dM_N(x) 
+ \frac{1}{2N}\iint \psi_{2,t} (x,y) \,dM_N(x)\,dM_N(y), 
\end{equation}
where (without loss of generality) we assume that $ \psi_{2,t} (x,y) = \psi_{2,t} (y,x)$.

Since we do not want to use gradient of functions but general vector fields (as this gives us more flexibility), in order to find an ansatz for
an approximate solution of \eqref{eq:laplace2} we compute first the gradient of $\psi$:
$$
\partial_i \psi^N_{t}=\psi_{0,t}'(\lambda_i)+\frac{1}{N}\psi_{1,t}'(\lambda_i)+
\frac{1}{N}\xi^N_{1,t}(\lambda_i,M_N),\qquad \xi^N_{1,t}(x,M_N):=\int \partial_1\psi_{2,t}(x,y)\,dM_N(y).
$$
This suggests us the following ansatz for the components of $\YY^N_t$:
\begin{equation}
\label{eq:YN}
\YY^N_{i,t}(\lambda_1,\ldots,\lambda_N):=\yy_{0,t}(\lambda_i)+\frac{1}{N}\yy_{1,t}(\lambda_i)+
\frac{1}{N}\xxi_{t}(\lambda_i,M_N),\quad \xxi_{t}(x,M_N):=\int \zz_t(x,y)\,dM_N(y),
\end{equation}
for some functions $\yy_{0,t},\yy_{1,t}:\R\to\R$, $\zz_t:\R^2\to \R$.

Here and in the following, given a function of two variables $\psi,$ we write $\psi\in C^{s,v}$ to denote that it is $s$ times continuously differentiable with respect to the first variable
and $v$ times with respect to the second.

The aim of this section is to prove the following result:
\begin{prop}\label{prop:key}Assume
$V',W\in C^{r}$ with $r\ge 30$.
Then, there exist  $\yy_{0,t}\in C^{r-2}$, $\yy_{1,t}\in C^{r-8}$, and $\zz_t\in C^{s,v}$ for $s+v \leq r-5$,
such that
$${\mathcal R}^N_t:=
\biggl( c_t^N-\beta\sum_{i<j}\frac{\YY^N_{i,t}-\YY^N_{j,t}}{\lambda_i-\lambda_j}
+N\sum_i W(\lambda_i) +N \sum_i V_t'(\lambda_i) \YY^N_{i,t}\biggr)-{\rm div} \YY^N_{t}$$
satisfies 
$$\|{\mathcal R}^N_t\|_{L^1({\mathbb P^N_{V_t}})}\le  C\,\frac{(\log N)^3 }{N}$$
for some positive constant $C$ independent of $t\in [0,1]$.
\end{prop}
The proof of this proposition is pretty involved, and will take the rest of the section.

\subsection{Finding an equation for $\yy_{0,t}, \yy_{1,t},\zz_{t}$.}

Using \eqref{eq:YN} we compute
$$
{\rm div}\YY^N_t=N\int \yy_{0,t}'(x)\,dL_N(x)+\int \yy_{1,t}'(\lambda)\,dL_N(x)
+\int \partial_1\xxi_{t}(x,M_N)\,dL_N(x)
+\eeta(L_N),
$$
where, given a measure $\nu$, we set
\begin{align*}
\eeta(\nu):=\int\partial_{2}\zz_{t}(y,y)\,d\nu(y).
\end{align*}
Therefore, recalling that $M_N=N (L_N-\mu_{V_t})$, we get
\begin{align*}
{\mathcal R}^N_t&=
-\frac{\beta N^2}{2}  \iint\frac{\yy_{0,t}(x)-\yy_{0,t}(y)}{x-y}\,dL_N(x)\,dL_N(y)
+N^2\int V'_t \,\yy_{0,t} \,dL_N+N^2 \int W \,dL_N\\
&
-\frac{\beta N}{2}  \iint\frac{\yy_{1,t}(x)-\yy_{1,t}(y)}{x-y}\,dL_N(x)\,dL_N(y)
+N\int V'_t \,\yy_{1,t} \,dL_N\\
&
-\frac{\beta N}{2}  \iint\frac{\xxi_t(x,M_N) - \xxi_t(y,M_N)}{x-y}\,dL_N(x)\,dL_N(y)
+N\int V'_t(x)\,\xxi_t(x,M_N)\, dL_N(x)\\
& -\frac{1}{N} \eeta(M_N) - N\Bigl(1-\frac{\beta}2\Bigr)  \int \yy_{0,t}'\,dL_N\\
& - \Bigl(1-\frac{\beta}2\Bigr)  \int \yy_{1,t}'\,dL_N
-\Bigl(1-\frac{\beta}2\Bigr)  \int \partial_1\xxi_{t}(x,M_N)\,dL_N(x)
-\eeta(\mu_{V_t})
+\tilde c^{N}_t, 
\end{align*}
where $\tilde c^N_t$ is a constant and we use the convention that, when we integrate a function of the form $\frac{f(x)-f(y)}{x-y}$
with respect to $L_N\otimes L_N$, the diagonal terms give $f'(x)$. 

We now observe that $L_N$ converges towards $\mu_{V_t}$ as $N\to \infty$ \cite{Boutet},  see also \cite{BAG, AGZ} for the corresponding large deviation principle, and the latter minimizes $I_{V_t}$ (see \eqref{entr}). Hence, considering
$\mu_\varepsilon:=(x+\varepsilon f)_\#\mu_{V_t}$ and  writing that $I_{V_t}(\mu_{\varepsilon})\ge I_{V_t}(\mu_{V_t})$, by taking the derivative with respect to $\varepsilon$
at $\varepsilon =0$ we get
\begin{equation}\label{bn}
\int V'_t (x)f(x)\,d\mu_{V_t}(x)=\frac{\beta}{2}\iint \frac{f(x)-f(y)}{x-y}\, d\mu_{V_t}(x)\,d\mu_{V_t}(y)\end{equation}
for all smooth bounded functions $f:\R\to\R$. Therefore we can recenter $L_N$ by $\mu_{V_t}$ in the formula above: more precisely, if we set
\begin{equation}\label{defXi}\Xi f(x):=-\beta \int\frac{f(x)-f(y)}{x-y}\, d\mu_{V_t}(y)+V'_t(x)f (x),\end{equation}
then
\begin{multline*}
N^2\int V_t'\, f \,dL_N- \frac{\beta N^2}{2}\iint \frac{f(x)-f(y)}{x-y}\,dL_N(x)\,dL_N(y)\\
=N\int \Xi f\, dM_N-\frac{\beta}{2}\iint \frac{f(x)-f(y)}{x-y}\,dM_N(x)\,dM_N(y)
\end{multline*}
Applying this identity
to $f=\yy_{0,t},\yy_{1,t},\xxi_t(\cdot,M_N)$ and recalling the definition of $\xxi_t(\cdot,M_N)$ (see \eqref{eq:YN}), we find
\begin{align*}
{\mathcal R}^N_t&= N \int  [\Xi \yy_{0,t}+W]\,dM_N\\
& +\int \left( \Xi\yy_{1,t}+\Big(\frac{\beta}2 - 1 \Big)\bigg[\yy_{0,t}' +\int \partial_{1}\zz_{t}(z,\cdot)d\mu_{V_t}(z)\biggr]\right)\,dM_N\\ 
&+\iint dM_N(x)\,dM_N(y)\left(\Xi \zz_{t}(\cdot,y)[x]-\frac{\beta}{2}\frac{\yy_{0,t}(x)-\yy_{0,t}(y)}{x-y}\right)+C^N_t +E_N,
\end{align*}
where
$$  \Xi \zz_t(\cdot, y) [x] =-\beta\int \frac{\zz_t(x,y)-\zz_t(\tilde x,y )}{x-\tilde x }\,d\mu_{V_t}(\tilde x)+V'_t(x)\zz_t(x,y),$$
$C^N_t$ is a deterministic term,
and $E_N$ is a remainder that we will prove to be negligible:
\begin{equation}
\label{eq:EN}
\begin{split}
E_N&:= -\frac{1}{N}\int \partial_{2}\zz_{t}(x,x)\,dM_N(x)-\frac{1}{N}\Bigl(1-\frac{\beta}2\Bigr) \int \yy_{1,t}'\,dM_N\\
&-\frac{1}{N}\Bigl(1-\frac{\beta}2\Bigr)\iint \partial_1\zz_t(x,y)\,dM_N(x)\,dM_N(y)\\
&-\frac{\beta}{2N}\iint \frac{\yy_{1,t}(x)-\yy_{1,t}(y)}{x-y}\, dM_N(x)\,dM_N(y)\\
&-\frac{\beta}{2N}\iiint\frac{\zz_t(x,y)-\zz_t(\tilde x,y)}{x-\tilde x}\,dM_N(x)\,dM_N(y)\,dM_N(\tilde x).
\end{split} 
\end{equation}
Hence, for ${\mathcal R}^N_t$  to be small we want to impose
\begin{equation}
\label{defpsi}
\begin{split}
\Xi\yy_{0,t}&=-W+c, \\
 \Xi \zz_t(\cdot, y) [x]& =-\frac{\beta}{2} \frac{\yy_{0,t}(x)-\yy_{0,t}(y)}{x-y}+c(y),\\
\Xi\yy_{1,t}&=-\Bigl(\frac{\beta}{2}-1\Bigr)\bigg[\yy_{0,t}' + \int \partial_{1}\zz_{t}(z,\cdot)\,d\mu_{V_t}(z)\biggr]+c',
\end{split}
\end{equation}
 where $c,c'$ are some constant to be fixed later, and $c(y)$ does not depend on $x$.
 \subsection{Inverting the operator $\Xi$.}
We now prove a key lemma, that will allow us to find the desired functions $\yy_{0,t},\yy_{1,t},\zz_{t}$.
\begin{lem} \label{central}
Given $V:\R\to \R$, assume that $\mu_{V}$ has support given by $[a,b]$
and that
$$
\frac{d\mu_V}{dx}(x)=S(x)\sqrt{(x-a)(b-x)}
$$
with $S(x)\geq \bar c>0$ a.e. on $[a,b]$. 

Let $g:\mathbb R\ra\mathbb R$ be a $C^k$ function and assume that $V$ is of class $C^p$. Set
$$\Xi f(x):=-\beta\int\frac{f(x)-f(y)}{x-y} d\mu_V(x) +V'(x)f(x)$$
Then there exists a unique constant $c_g$ such that
the equation
$$\Xi f(x)=g(x)+c_g$$
has a  solution of class $C^{(k-2)\wedge (p-3)}$.  More precisely, for $j\le (k-2)\wedge (p-3)$ there is a finite constant $C_j$ such that
\begin{equation}\label{toto2b}
\|f\|_{C^j(\R)}\le 
 C_j \|g\|_{C^{j+2}(\R)} ,
\end{equation}
where, for a function $h$, $\|h\|_{C^j(\R)}:=\sum_{r=0}^j\|h^{(r)}\|_{L^\infty(\mathbb R)}$.

Moreover  $f$ (and its derivatives) behaves like $(g(x)+c_g)/V'(x)$ (and its corresponding derivatives) when $|x|\to +\infty$.

This solution will be denoted by $\Xi^{-1}g$.
\end{lem}
Note that $Lf(x)=\Xi f'(x)$ can be seen as the asymptotics of the infinitesimal generator of the Dyson Brownian motion taken in the set where
the spectral measure approximates $\mu_V$. This operator is central in our approach, as much as the Dyson Brownian motion is central to prove universality in e.g.
\cite{ESYY,BEY1,BEY2}.
\begin{proof}
As a consequence of \eqref{bn}, we have
\begin{equation}
\label{eq:muV V'}
\beta\,PV\int\frac{1}{x-y} \,d\mu_{V}(y)=V'(x)\qquad \text{on the support of $\mu_{V}$.}
\end{equation}
Therefore  the equation $\Xi f(x)=g(x)+c_g$ on the support of $\mu_{V}$ amounts to
\begin{equation}
\label{eq:tri}
\beta\,PV \int \frac{f(y)}{x-y} \,d\mu_{V}(y)= g(x)+c_g\qquad \forall\,x \in [a,b].
\end{equation}
Let us write $$d(x):=d\mu_{V}/dx =S(x)\sqrt{(x-a)(b-x)}$$ with $S$  positive inside the support $[a,b]$.
We claim that
$S\in C^{p-3}([a,b])$.

Indeed, by \eqref{bn} with $f(x)=(z-x)^{-1}$ for $z\in [a,b]^c$,  we find that the Stieltjes transform $G(z)=\int (z-y)^{-1}\,d\mu_{V}(y)$ 
 satisfies, for $z$ outside $[a,b]$,
 $$\frac{\beta}{2}G(z)^2=G(z)V'(\Re(z))+F(z),\qquad\mbox{ with } F(z)=\int \frac{V'(y)-V'(\Re(z))}{z-y}\,d\mu_{V}(y)\,.$$
Solving this quadratic equation so that $G\to 0$  as $|z|\to\infty$  yields
\begin{equation}
\label{eq:G}
G(z)=\frac{1}{\beta} \Big(V'(\Re(z))-\sqrt{ [V'(\Re(z))]^2+2\beta F(z)}\Big).
\end{equation}
Notice that $V'(\Re(z))^2+2\beta F(z)$ becomes real as $z$ goes to the real axis, and negative inside $[a,b]$.
Hence, since $-\pi^{-1} \Im G(z)$
converges to the density of $\mu_V$ 
as $z$ goes to the real axis (see e.g \cite[Theorem 2.4.3]{AGZ}),
we get
\begin{equation}
\label{eq:S}
-S(x)^2(x-a)(b-x)=(\beta\pi)^{-2}\left[V'(x)^2+2\beta F(x)\right].
\end{equation}
This implies in particular that
$\{a,b\}$ are the two points of the real line where $V'(x)^2+2\beta F(x)$ vanishes. 
Moreover $F(x)=-\int \int_0^1 V''(\alpha y+(1-\alpha)x) \,d\alpha\,d\mu_{V}(y)$ is of class $C^{p-2}$
on $\mathbb R$ (recall that $V\in C^p$ by assumption), therefore $(V')^2+2\beta F\in C^{p-2}(\R)$.
Since we assumed that $S$ does not vanish in  $[a,b]$, from \eqref{eq:S} we deduce that $S$ is of class
 $C^{p-3}$ on $[a,b]$. 
 
To solve \eqref{eq:tri} we apply
Tricomi's formula \cite[formula 12, p.181]{tricomi}
and we find that, for $x\in [a,b]$,
$$\beta f(x) {\sqrt{(x-a)(b-x)}}  d(x)= PV \int_{a}^{b} \frac{{\sqrt {(y-a)(b-y)}}}{{y-x}} (g(y) +c_g)dy + c_2 := h(x) $$
for some constant $c_2$, hence
$$
   \begin{array}{rcl}
    h(x) &=& \beta f(x){{(x-a)(b-x)}}  S(x)\\
     &=& PV \int_a^b \frac{{\sqrt {(y-a)(b-y)}}}{{y-x}} (g(y) +c_g)dy + c_2\\
     &=&PV \int_a^b {\sqrt {(y-a)(b-y)}} \frac{g(y) -g(x)}{y-x}dy + (g(x)+ c_g)PV \int_a^b \frac{{\sqrt {(y-a)(b-y)}}}{{y-x}} dy + c_2\\
     &=&  \int_a^b {\sqrt {(y-a)(b-y)}} \frac{g(y) -g(x)}{y-x}dy -\pi \left(x-\frac{a+b}{2} \right)(g(x)+ c_g) + c_2,
   \end{array}
$$
where we used  that, for $x\in [a,b]$,
$$
PV \int_a^b \frac{{\sqrt {(y-a)(b-y)}}}{{y-x}}dy = -\pi \Bigl(x -\frac{ a+b}{2}\Bigr).
$$ 
Set
$$ h_0 (x)=  \int_a^b {\sqrt {(y-a)(b-y)}}\, \frac{g(y) -g(x)}{y-x}dy.$$
Then  $h_0$ is of class $C^{k-1}$ (recall that $g$ is of class $C^k$). We next  choose $c_g$ and $c_2$ such that $h$ vanishes at $a$ and $b$ (notice that this choice uniquely identifies $c_g$).

We note  that $f \in C^{(k-2)\wedge (p-3)}([a,b])$.
Moreover, we can bound its derivatives 
in terms of the derivatives of $h_0, g$ and $S$: if we assume $j\le p-3$, we find that there exists a constant $C_j$, which depends only on the derivatives of $S$, such that
$$
\|f^{(j)}\|_{L^\infty([a,b])}\le C_j \max_{p\le j } \Bigl(\|h_0^{(p+1)}\|_{L^\infty([a,b])}+\|g^{(p+1)}\|_{L^\infty([a,b])}\Bigr)\le C_j \max_{p\le j+2} \|g^{(p)}\|_{L^\infty([a,b])}.
$$
Let us define $$k(x):= \beta\,PV\int \frac{f(y)}{x-y} d\mu_{V}(y)-g(x)-c_g\qquad \forall\,x \in \R.$$
By \eqref{eq:tri} we see that $k\equiv 0$ on $[a,b]$.
To ensure that $\Xi f=g+c_g$ also outside the support of $\mu_{V}$ we want
$$f(x)\biggl(\beta\,PV\int\frac{1}{x-y}\,d\mu_{V}(y)-V'(x)\biggr)=k(x)\qquad\forall\,x \in [a,b]^c.$$
Let us consider the function $\ell:\R\to\R$ defined as
\begin{equation}
\label{eq:ell}
\ell(x):= \beta\,PV\int\frac{1}{x-y}\,d\mu_V(y)-V'(x).
\end{equation}
Notice that, thanks to \eqref{eq:G}, $\ell(x)=\beta G(x) - V'(x)=-\beta\sqrt{ [V'(x)]^2+2\beta F(x)}$.
Hence, comparing this expression with \eqref{eq:S} and recalling that $S \geq \bar c>0$ in $[a,b]$, 
we deduce that $[V'(x)]^2+2\beta F(x)$ is smooth and  has simple zeroes both at $a$ and $b$, therefore $[V'(x)]^2+2\beta F(x)>0$ in
$[a-\eps,b+\eps]\backslash [a,b]$ for some $\eps>0$.

This shows that
$\ell$ does not vanish in $[a-\eps,b+\eps]\backslash [a,b]$.
Recalling that we can freely modify $V$ outside $[a-\eps,b+\eps]$ (see proof of Lemma \ref{propa}), we can actually assume that
$\ell$ vanishes
at $\{a,b\}$
and does not vanish in the whole $[a,b]^c$.

We claim that $\ell$ is H\"older $1/2$ at the boundary points, and in fact is equivalent to a square root there.
Indeed, it is immediate to check that $\ell$ is of class $C^{p-1}$ except possibly at the boundary points $\{a,b\}$. 
Moreover
\begin{align*}
PV\int\frac{1}{x-y}\,d\mu_{V}(y)&= S(a) \int_a^b\frac{1}{x-y} \sqrt{(y-a)(b-y)} \,dy\\
&+\int_a^b \frac{y-a}{x-y}\biggl(\int_0^1 S'(\alpha a+(1-\alpha) y) d\alpha\biggr) \sqrt{(y-a)(b-y)} \,dy.
\end{align*}
The first term can be computed exactly and we have, for some $c\neq 0$,
\begin{equation}\label{ipo}
\int_a^b\frac{1}{x-y} \sqrt{(y-a)(b-y)} \,dy =c (b-a)\biggl( \frac{x-\frac{a+b}{2}}{b-a} -\sqrt{ \Bigl(\frac{x-\frac{a+b}{2}}{b-a}\Bigr)^2-\frac{1}{4}}\biggr)\end{equation}
which is H\"older $1/2$, and in fact behaves as a square root at the boundary points. On the other hand, since $S$ is of class $C^{p-3}$ on $[a,b]$ with $p\ge 4$,  the second function is differentiable,  with derivative
at $a$ given by
$$\int_a^b \frac{1}{a-y}\biggl(\int_0^1 S'(\alpha a+(1-\alpha) y) d\alpha \biggr)\,\sqrt{(y-a)(b-y)} \,dy,$$
which is a convergent integral. The claim follows. 

Thus, for $x$ outside the support of $\mu_{V}$ we can set
$$f(x):=\ell(x)^{-1}k(x).$$ 
With this choice $\Xi f=g + c_g$ and $f$  is of class $C^{(k-2)\wedge (p-3)}$ on $\mathbb{R}\setminus\lbrace a,b \rbrace$.

We now want to show that $f$ is of class $C^{(k-2)\wedge (p-3)}$ on the whole $\R$.
For this we need to check the continuity of $f$ and its derivatives at the boundary points, say at $a$ (the case of $b$ being similar).
We take hereafter $r\le ( k-2)\wedge (p-3)$, so that $f$ has $r$ derivatives inside $[a,b]$ according to the above considerations. \\
Let us first deduce the continuity of $f$ at $a$.
We write, with $f(a^+)=\lim_{x\downarrow a} f(x)$,
$$k(x)=f(a^+) \ell(x)+ k_1(x)$$
with
$$
k_1(x):=\beta\left( PV\int  \frac{f(y)}{x-y} d\mu_{V}(y)- PV\int \frac{f(a^+)}{x-y} d\mu_V(y) \right) +g(x)+c_g+ f(a^+)V'_t(x).
$$
Notice that since $f=\ell^{-1}k$ outside $[a,b]$, if we can show that 
$\ell^{-1}(x)k_1(x)\to 0$ as $x\uparrow a$ then we would get $f(a^-)=f(a^+)$,
proving the desired continuity.

To prove it we first notice that $k_1$ vanishes at $a$  (since both $k$
and $\ell$ vanish inside $[a,b]$), hence
\begin{align*}
k_1(x)&= \beta\left( PV\int  \frac{f(y)-f(a^+) }{x-y} d\mu_{V}(y)- PV\int \frac{f(y)-f(a^+)}{a-y} d\mu_{V}(y) \right)  +\tilde g(x)-\tilde g(a)\\
&=\beta (a-x)\,PV\int \frac{f(y)-f(a^+)}{(x-y)(a-y)} d\mu_{V}(y) +\tilde g(x)-\tilde g(a),
\end{align*}
with $\tilde g:=g+f(a^+)V' \in C^1$.
Assume $1\le ( k-2)\wedge (p-3)$. Since $f$ is of class $C^1$ inside $[a,b]$
we have $|f(y)-f(a^+)|\leq C|y-a|$, from which we deduce that 
$|k_1(x)|\leq C|x-a|$ for $x \leq a$.

Hence $\ell^{-1}(x)k_1(x)\to 0$ as $x\uparrow a$
(recall that $\ell$ behaves as a square root near $a$),
which proves that $$\lim_{x\uparrow a} f(x)=\lim_{x\downarrow a}f(x)$$
and shows the continuity of $f$ at $a$.

We now consider the next derivative: we write
$$k(x)=\bigl[f(a)+f'(a^+) (x-a)\bigr] \ell(x)+ k_2(x)$$
with
\begin{align*}
k_2(x)&:= \beta (a-x)\,PV\int \frac{f(y)-f(a^+)-(y-a) f'(a^+)}{(x-y)(a-y)}\, d\mu_{V}(y)\\
& +\tilde g(x)-\tilde g(a)+f'(a^+) (x-a) V'_t(x).
\end{align*}
Since $k=\ell\equiv 0$ on $[a,b]$ we have $k_2(a)=k_2'(a^+)=k_2'(a^-)=0$.
Hence, since $f$ is of class $C^2$ on $[a,b]$, we see that $|k_2(x)|\leq C|x-a|^2$ for $x \leq a$, therefore  $k_2(x)/\ell(x)$ is of order $|x-a|^{3/2}$, thus
$$
f(x)=f(a)+f'(a^+) (x-a)+ O(|x-a|^{3/2})\qquad \text{for $x \leq a$,}
$$
which shows that $f$ has also a continuous derivative.

We obtain the continuity of the next derivatives similarly. Moreover, away from the boundary point the $j$-th derivative of $f$ outside
$[a,b]$ is of the same order than that of $g/ V'$, while near the boundary points it is governed by  the derivatives of $g$ nearby, therefore
\begin{equation}\label{toto2}
\|f^{(j)}\|_{{L^\infty([a,b]^c)}}\le 
 C'_j\max_{r\le j+2} \|g^{(r)}\|_{L^\infty(\mathbb R)}  \,.
\end{equation}
Finally, it is clear that $f$ behaves like $(g(x)+c_g)/V'(x)$ when $x$ goes to infinity.
\end{proof}
\def\La{{\mathcal L}}
\subsection{Defining the functions $ \yy_{0,t},\yy_{1,t},\zz_{t}$}
To define the functions $\yy_{0,t},\yy_{1,t},\zz_{t}$ according to \eqref{defpsi}, notice that Lemma \ref{propa} shows that
the hypothesis of Lemma \ref{central} are fulfilled.
Hence, as a consequence of Lemma \ref{central} we find the following result
(recall that $\psi\in C^{s,v}$ means that $\psi$ is $s$ times continuously differentiable with respect to the first variable
and $v$ times with respect to the second).
\begin{lem}\label{lem:regular} Let $r\ge 7$.
If $W, V' \in C^r$, we can choose $\yy_{0,t}$ of class $C^{r-2}$,  $\zz_t\in C^{s,v}$ for $s+v\le r-5,$ and $\yy_{1,t}\in C^{r-8}$. Moreover,
these functions (and their derivatives) go to zero at infinity like  $1/V'$
(and its corresponding derivatives).
\end{lem}
\begin{proof} By Lemma \ref{central} we have $\yy_{0,t}=\Xi^{-1} W\in C^{r-2}$. 
For $\zz_{t}$, we can rewrite
\begin{align*}
\Xi \zz_t(\cdot, y) [x] &=-\frac{\beta}{2}\int_0^1  \yy_{0,t}'(\alpha x +(1-\alpha)y )\,d\alpha +c(y)\\
&= -\frac{\beta}{2}\int_0^1  [\yy_{0,t}'(\alpha x +(1-\alpha)y )+c_\alpha(y)]\,d\alpha
\end{align*}
where we choose $c_\alpha(y)$ to be the unique constant provided by Lemma \ref{central} which ensures that $\Xi^{-1} [\yy_{0,t}'(\alpha x +(1-\alpha)y )+c_\alpha(y)]$ is smooth.
This gives that $c(y)=\int_0^1 c_\alpha(y) d\alpha$. Since $\Xi^{-1}$ is a linear integral operator,
we have
$$\zz_t(x,y)=
-\frac{\beta}{2}\int_0^1 \Xi^{-1}[ \yy_{0,t}'(\alpha \cdot +(1-\alpha)y )](x)\, d\alpha\,.$$
As the variable $y$ is only a translation, it is not difficult to check that $\zz_t\in C^{s,v}$  for any $s+v\le r-5$. It follows that 
$$-\Bigl(\frac{\beta}{2}-1\Bigr)\bigg[\yy_{0,t}' + \int \partial_{1}\zz_{t}(z,\cdot )\,d\mu_{V_t}(z)\biggr]+c'$$
is of class $C^{r-6}$ and therefore by Lemma \ref{central} we can choose $\yy_{1,t} \in C^{r-8}$, as desired.

The decay at infinity is finally again a consequence of Lemma \ref{central}.
 \end{proof}

\subsection{Getting rid of the random error term $E_N$}
\label{sect:error}

We show that   the $L^1_{\mathbb P_{V_t}^N}$-norm of  the error term $E_N$ defined in \eqref{eq:EN} goes
to zero. This could be easily derived from \cite{BEY1}, but we here provide a self-contained proof.
To this end, we first make some general consideration
on the growth of variances.

Following e.g. \cite[Theorem 1.6]{Edouard-mylene}, up to assume that $V_t$ goes sufficiently fast  at infinity (which we did, see Lemma \ref{propa}),  it is easy to show that
  there exists a constant $\tau_0>0$ so that for all $\tau \ge \tau_0$,
$$
\mathbb P^N_{V_t}\biggl( D( L_N,\mu_{V_t})\ge \tau\sqrt{\frac{\log N}{N}} \biggr)\le e^{-c \tau^2 N\log N }\,,
$$
where $D$ is the $1$-Wasserstein distance
$$
D(\mu,\nu):=\sup_{\|f'\|_{\infty}\le 1} \biggl|\int f(d\mu-d\nu)\biggr|.
$$ 
Since $M_N=N(L_N-\mu_{V_t})$ we get
\begin{equation}
\label{eq:W M}
D( L_N,\mu_{V_t})=\frac{1}{N} \sup_{\|f'\|_{\infty}\le 1} \biggl|\int f\,dM_N\biggr|,
\end{equation}
hence for $\tau\ge \tau_0$
\begin{equation}\label{conc}
\mathbb P^N_{V_t}\biggl( \sup_{\|f\|_{\rm Lip}\le 1} \biggl|\int f\,dM_N\biggr| \ge \tau\sqrt{N \log N} \biggr)\le e^{-c \tau^2 N\log N }.\end{equation}
This already shows that, if $f$ is sufficiently smooth, $\int f(x,y) dM_N(x)\,dM_N(y)$ is of order at most $N\log N$. More precisely
$$\int f(x,y)\, dM_N(x)\,dM_N(y)=\int \hat f(\zeta,\xi)\biggl( \int e^{i\zeta x} dM_N(x) \int e^{i\xi x} dM_N(x)\biggr)\, d\xi\, d\zeta,$$
so that with probability greater than $1-e^{-c\tau_0^2N \log N}$ we have
\begin{equation}\label{bor1}
\left|\int f(x,y) \,dM_N(x)\,dM_N(y)\right|\le \tau_0^2\, N\log N \,\int | \hat f(\zeta,\xi)|\, |\zeta|\,|\xi|\,d\zeta\, d\xi\,.\end{equation}

To improve this estimate, we shall use loop equations as well as Lemma
\ref{central}. Given a function $g$ and a measure $\nu$, we use the notation $\nu(g):=\int g\,d\nu$.
\begin{lem}\label{conch} Let $g$ be a smooth function.
Then, if $\tilde M_N=NL_N-N\mathbb E_{V_t}[ L_N]$, 
there exists a finite constant $C$ such that
\begin{align*}
\sigma^{(1)}_N(g)&:=\biggl| \int M_N(g) \,d \mathbb P^N_{V_t}\biggr|\le C \,m(g)
=:B^1_N(g)\\
\sigma^{(2)}_N(g)&:= \int \left(\tilde M_N(g)\right)^2  d \mathbb P^N_{V_t}\le C\Bigl(m(g)^2 + m(g)\|g\|_\infty +\|\Xi^{-1} g\|_\infty \|g'\|_{\infty}\Bigr)=:B^2_N(g)
 \\
\sigma^{(4)}_N(g)&:=\int \left(\tilde M_N(g)\right)^4 \, d \mathbb P^N_{V_t}\\
&\le  C \Bigl( \|\Xi^{-1} g\|_\infty \|g'\|_{\infty} \sigma_N^{(2)}(g)+ \|g\|_\infty^3m(g)
+m(g)^2 \sigma_N^{(2)}(g)+m(g)^4\Bigr) 
=:B^4_N(g),
\end{align*}
where 
\begin{align*}
m(g)&:=\Bigl|1-\frac{\beta}{2}\Bigr|\|(\Xi^{-1} g)'\|_\infty + \frac{\beta}{2}
  \log N \int  |\hat \Xi^{-1} g|(\xi) \,|\xi|^3 \,d\xi.
\end{align*}
\end{lem}
\begin{proof}
 First observe that, by integration by parts, for any $C^1$ function $f$
 \begin{equation}\label{loop1}
 \int \biggl(N\sum_i V'(\lambda_i)f(\lambda_i)- \beta \sum_{i<j}\frac{f(\lambda_i)-f(\lambda_j)}{\lambda_i-\lambda_j} \biggr)\,d\mathbb P^N_{V_t}=\int \sum_{i} f'(\lambda_i) \,d\mathbb P^N_{V_t}\end{equation}
 which we can rewrite as the first loop equation
 \begin{equation}
 \label{lkjh}
 \int   M_N(\Xi f)\,d \mathbb P^N_{V_t}=\int \biggl[ \Bigl(1-\frac{\beta}{2}\Bigr)\int f' dL_N + \frac{\beta}{2 N} \int \frac{f(x)-f(y)}{x-y} d M_N(x) d M_N(y)\biggr]\,d\mathbb P^N_{V_t}\,.\end{equation}
We denote 
 $$F_N(g):=\Bigl(1-\frac{\beta}{2}\Bigr)\int (\Xi^{-1} g)' \,dL_N + \frac{\beta}{2 N} \int \frac{\Xi^{-1} g(x)-\Xi^{-1} g(y)}{x-y} \,d M_N(x)\, d M_N(y)$$
 so that taking $f:=\Xi^{-1} g$  in \eqref{lkjh} we deduce
$$\int   M_N(g)\,d \mathbb P^N_{V_t}=\int F_N(g)\, d \mathbb P^N_{V_t}.$$
To bound the right hand side above, we notice that 
$\Xi^{-1}g$ goes to zero at infinity like $1/V'$ (see Lemma \ref{central}). Hence we can write its Fourier transform and get
 \begin{multline*}
 \int \frac{\Xi^{-1} g(x)-\Xi^{-1} g(y)}{x-y} \,d M_N(x)\, d M_N(y)\\
 = i\int  d\xi\, \xi  \,\hat \Xi^{-1}g(\xi)\int_0^1  d\alpha \int e^{i\alpha\xi x}\,dM_N(x)
 \int e^{i(1-\alpha)\xi y}\,dM_N(y),
 \end{multline*}
 so that we deduce (recall \eqref{eq:W M})
 $$\sup_{D(L_N,\mu_{V_t})\le \tau_0 \sqrt{\log N/N}} F_N(g)\le (1+ \tau_0^2)\, m(g).$$
 On the other hand, as the mass of $M_N$ is always bounded by $2N$, we deduce that 
  $F_N(g)$ is bounded everywhere  by $ N  m(g) 
 $. Since the set  $\{D( L_N,\mu_{V_t})\ge \tau_0 \sqrt{N\log N}\}$
 has small probability (see \eqref{conc}), we conclude that
 \begin{equation}\label{bor2}
\biggl|\int M_N(g) \,d \mathbb P^N_{V_t}\biggr|\le  
N e^{-c\tau_0^2  N\log N }m(g) + (1+\tau_0^2) m(g) \leq C\,m(g),\end{equation}
which proves our first bound.

Before proving the next estimates, let us make a simple remark: using the definition of $M_N$
and $\tilde M_N$ it is easy to check that, for any function $g$,
\begin{equation}
\label{eq:difference MN tilde}
\left|M_N(g) - \tilde M_N(g)\right| =\biggl|\int M_N(g) \,d \mathbb P^N_{V_t}\biggr|.
\end{equation}

To get estimates on the covariance we obtain the second loop equation by changing $V(x)$ into $V(x)+\delta \,g(x)$ in \eqref{loop1} and differentiating with respect to $\delta$ at $\delta=0$.
This gives 
\begin{equation}
\label{lkjp}
\begin{split}
&\int M_N(\Xi f) \tilde M_N(g) \,d \mathbb P^N_{V_t}=\int L_N(f g')\,d \mathbb P^N_{V_t} \\
&+\int \biggl[ \Bigl(1-\frac{\beta}{2}\Bigr)\int f' dL_N + \frac{\beta}{2 N} \int \frac{f(x)-f(y)}{x-y} d M_N(x) d M_N(y)\biggr]\tilde M_N(g)\,
 d\mathbb P^N_{V_t}.
 \end{split}
 \end{equation}
 We now notice that $M_N(\Xi f)-\tilde M_N(\Xi f)$ is deterministic and $\int \tilde M_N(g) \,d \mathbb P^N_{V_t}=0$,
 hence the left hand side is equal to
 $$
 \int \tilde M_N(\Xi f) \tilde M_N(g) \,d \mathbb P^N_{V_t}.
 $$
We take $f:=\Xi^{-1}g$
and we argue similarly to above (that is, splitting the estimate depending whether 
$D( L_N,\mu_{V_t}) \geq \tau_0\sqrt{N\log N}$ or not,
and use that $|\tilde M_N(g)|\leq N\|g\|_\infty$)
  to deduce that $\sigma_N^{(2)}(g):=\int |\tilde M_N( f) |^2 d \mathbb P^N_{V_t}$ satisfies
\begin{equation}
 \label{bo1}
 \begin{split}
\sigma_N^{(2)}(g)&\le \|g'\Xi^{-1} g\|_\infty  +\int |F_N(g)| |\tilde M_N(g)| d\mathbb P^N_{V_t}\\
&\le \|\Xi^{-1} g\|_\infty\,\|g'\|_{\infty} + N^2 e^{-c\tau_0^2 N\log N}\|g\|_\infty m(g)+
C\,m(g)\int |\tilde M_N(g)| d\mathbb P^N_{V_t}\\
& = \|\Xi^{-1} g\|_\infty \|g'\|_{\infty} +N^2 e^{-c\tau_0^2 N\log N }m(g)\|g\|_\infty + C\,m(g)\sigma_N^{(2)}(g)^{1/2}.
\end{split}
\end{equation}
Solving this quadratic inequality yields
$$\sigma_N^{(2)}(g)\le C \Bigl[ m(g)^2 + m(g)\|g\|_\infty +\|\Xi^{-1} g\|_\infty \|g'\|_{\infty} \Bigr]
$$
for some finite constant $C$.

We finally turn to the fourth moment.
If we make an infinitesimal change of potential $V(x)$ into $V(x)+\delta_1\, g_2(x) +\delta_2\, g_3(x)$ and differentiate at $\delta_1=\delta_2=0$
 into \eqref{lkjp}
we get, denoting $g=g_1$, 
\begin{equation}
\label{lkj2}
\begin{split}
&\int M_N(\Xi f) \tilde M_N(g_1) \tilde M_N(g_2) \tilde M_N(g_3)\,d \mathbb P^N_{V_t}=\int \biggl[\sum_\sigma L_N(f g_{\sigma(1)}')\tilde M_N(g_{\sigma(2)}) \tilde M_N(g_{\sigma(3)}) \biggr]\,d\mathbb P^N_{V_t}+\\
&\int \biggl[ \Bigl(1-\frac{\beta}{2}\Bigr)\int f' \,dL_N + \frac{\beta}{2 N} \int \frac{f(x)-f(y)}{x-y}\, d M_N(x)\, d M_N(y)\biggr] M_N(g_1)\tilde M_N(g_2)\tilde M_N(g_3)
 \,d\mathbb P^N_{V_t} ,
\end{split}
\end{equation}
where we sum over the permutation $\sigma$ of $\{1,2,3\}$. Taking $\Xi f=g_1=g_2=g_3=g$,
by \eqref{eq:difference MN tilde}, \eqref{bor2}, and Cauchy-Schwarz inequality we get 
$$\sigma_{N}^{(4)}(g)\le C\Bigl[     \|g'\Xi^{-1} g\|_\infty \sigma_N^{(2)}(g) + \|g\|_\infty^3m(g)+m(g) \sigma_N^{(4)}(g)^{3/4}+ m(g)^2\sigma_N^{(2)}(g)\Bigr],
$$
which implies
$$\sigma_{N}^{(4)}(g)\le C\Bigl[   \|g'\Xi^{-1} g \|_\infty \sigma_N^{(2)}(g)+ \|g\|_\infty^3m(g)
+m(g)^2 \sigma_N^{(2)}(g)+m(g)^4\Bigr].$$
\end{proof}

Applying the above result with $g=e^{i\lambda \cdot}$ we get the following:
\begin{cor}
Assume that $V',W \in C^r$ with $r\ge 8$. Then  there exists a finite constant $C$ such that, for all $\lambda\in \mathbb R$,
\begin{align}
\int |M_N(e^{i\lambda\cdot})|^2  d \mathbb P^N_{V_t}&\le  C[\log N(1+|\lambda|^{7})]^2\,,\label{conc3b}\\
\int |M_N(e^{i\lambda\cdot})|^4  d \mathbb P^N_{V_t}&\le  C[\log N(1+|\lambda|^{7})]^4\,.\label{conc3bb}
\end{align}

\end{cor}
\begin{proof}
In the case $g(x)=e^{i\lambda x}$ we estimate the norms of $\Xi^{-1}g$ by using Lemma \ref{central},
and we get a finite constant $C$ such that 
$$\|\Xi^{-1} g\|_\infty\le C |\lambda|^2,\qquad \|\Xi^{-1}g'\|_\infty\le C |\lambda|^3,$$
whereas, since $\Xi^{-1}g$ goes fast to zero at infinity (as $1/V'$), for $j \leq r-3$ we have 
(see Lemma \ref{central})
$$|\hat\Xi^{-1} g|(\xi)\le C\frac{\| \Xi^{-1} g\|_{C^j}}{1+|\xi|^j}\le  C'\frac{\| g\|_{C^{j+2}}}{1+|\xi|^j}\le C'\frac{1+ |\lambda|^{j+2}}{1+|\xi|^j}.$$
Hence, we deduce that there exists a finite constant $C'$ such that
\begin{align*}
m(g)&\le C\log N \biggl( |\lambda|^3+1+ \int d\xi \,\frac{1+|\lambda|^{7}}{1+ |\xi|^{5}}|\xi|^3 \biggr)=C'\log N\left(1+|\lambda|^7\right),\\
B_N^1(g)&\le C'\log N \left(1+|\lambda|^7\right),\\
B_N^2(g)&\le C' (\log N)^2 \left(1+|\lambda|^7\right)^2,\\
B_N^4(g)&\le C'(\log N)^4\left(1+|\lambda|^7\right)^4.\\
\end{align*}
Finally, for $k=2,4$, using \eqref{eq:difference MN tilde} and \eqref{bor2} we have
$$\int |M_N(e^{i\lambda\cdot})|^k \, d \mathbb P^N_{V_t}\le 2^{k-1}\left( \int |\tilde M_N(e^{i\lambda\cdot})|^k \,d \mathbb P^N_{V_t}
+\bigl(B_N^1(g)\bigr)^k\right)$$
from which the result follows.

\end{proof}
We can now estimate $E_N$. 

The linear term can be handled in the same way as we shall do now for the quadratic and cubic
terms (which are actually more delicate), so we just focus on them.

We have two quadratic terms in $M_N$  which sum up into 
$$E^1_N=-\frac{1}{N}\Bigl(1-\frac\beta2\Bigr)\iint \partial_1\zz_t(x,y)\,dM_N(x)\,dM_N(y)
-\frac{\beta}{2N}\iint \frac{\yy_{1,t}(x)-\yy_{1,t}(y)}{x-y} \,dM_N(x)\,dM_N(y).
$$
Writing
$$
\frac{\yy_{1,t}(x)-\yy_{1,t}(y)}{x-y}=\int_0^1 \yy_{1,t}'(\alpha x+(1-\alpha)y)\,d\alpha 
=\int_0^1\biggl( \int \widehat{\yy_{1,t}'}(\xi) e^{i(\alpha x+(1-\alpha)y)\xi}d\xi\biggr)\,d\alpha 
$$
we see that
$$
\iint \frac{\yy_{1,t}(x)-\yy_{1,t}(y)}{x-y} \,dM_N(x)\,dM_N(y)=\int d\xi\, \widehat{\yy_{1,t}'}(\xi)\int_0^1d\alpha \,M_N(e^{i\alpha\xi \cdot})M_N(e^{i(1-\alpha)\xi \cdot}),
$$
so using \eqref{conc3b} we get
\begin{multline*}
\int |E^1_N|  \,d \mathbb P^N_{V_t}\le  C\,\frac{(\log N)^2 }{N}\bigg(\int d\xi\, |\hat \yy_{1,t}|(\xi)\, |\xi|\, \bigl(1+|\xi|^7\bigr)^2\\
+ \iint d\xi \,d\zeta\, |\hat\zz_{t}| (\xi,\zeta) \,|\xi|\left(1+|\xi|^7\right) \left(1+|\zeta|^7\right) \bigg).
\end{multline*}
It is easy to see that the right hand side is finite if $\yy_{1,t}$ and $\zz_{t}$ are smooth enough
(recall that these functions and their derivatives decay fast at infinity). More precisely, to ensure that
$$
|\hat \yy_{1,t}|(\xi)\, |\xi|\, \bigl(1+|\xi|^7\bigr)^2 \leq \frac{C}{1+|\xi|^2} \in L^1(\R)
$$
and
$$
|\hat\zz_{t}| (\xi,\zeta) \,|\xi|\left(1+|\xi|^7\right) \left(1+|\zeta|^7\right) \leq \frac{C}{1+|\xi|^3+|\zeta|^3}\in L^1(\R^2),
$$
we need  $\yy_{1,t}\in C^{17}$ and 
$\zz_{t}\in C^{11, 7}\cap C^{8,10}$, so (recalling Lemma \ref{lem:regular}) $V',W\in C^{25}$ is  enough to guarantee that the right hand side is finite.

Using \eqref{conc3b}, \eqref{conc3bb}, and H\"older inequality, we can similarly bound the expectation of the  cubic term
\begin{align*}
E^2_N&=\frac{\beta}{2N}\iiint\frac{\zz_t(x,y)-\zz_t(\tilde x,y)}{x-\tilde x}\,dM_N(x)\,dM_N(y)\,dM_N(\tilde x)\\
&=i\frac{\beta}{2N}\iint d\xi \,d\zeta \,\widehat{\partial_1\zz_{t}}(\xi,\zeta)\int_0^1 d\alpha \, M_N(e^{i\alpha\xi\cdot})M_N(e^{i(1-\alpha)\xi\cdot})M_N(e^{i\zeta\cdot})
\end{align*}
to get
$$ \int |E^2_N|  \,d \mathbb P^N_{V_t}\le  C\,\frac{(\log N)^3 }{{N}}\iint d\xi\, d\zeta\,  |\hat\zz_{t}(\xi,\zeta)| \,|\xi|\left(1+|\xi|^7\right)^2 
\left(1+|\zeta|^7\right).$$
Again the right hand side is finite if $\zz_{t} \in C^{18, 7}\cap C^{15,10}$, which
 is ensured by Lemma \ref{lem:regular} if $V',W$ are of class $C^{30}$.

\subsection{Control on the deterministic term  $C^N_t$}
By what we proved above we have
$$
\int| {\mathcal R}^N_t - C^N_t|\,d{\mathbb P^N_{V_t}}\leq C\,\frac{(\log N)^3 }{{N}},
$$
thus, in particular,
$$
\bigl|C^N_t- \mathbb E[{\mathcal R}^N_t]\bigr|\leq C\,\frac{(\log N)^3 }{{N}}.
$$
Notice now that, by construction,
$${\mathcal R}^N_t =-\La \YY^N_t +N\sum_i W(\lambda_i)+c^N_t$$
with $c^N_t=-\mathbb E[ N\sum_i W(\lambda_i)]$ and
$$\La \YY:={\rm div}\YY+\beta\sum_{i<j}\frac{\YY^i-\YY^j}{\lambda_i-\lambda_j}-N\sum_i V'(\lambda_i) \YY^i,$$
and an integration by parts shows that, under $P_{V}^N$, $\mathbb E[\La \YY]=0$ for any vector field $\YY$.
This implies that $\mathbb E[{\mathcal R}^N_t]=0$, therefore
$|C^N_t|\leq  C\,\frac{(\log N)^3 }{{N}}$.

This concludes the proof of Proposition \ref{prop:key}.

\section{Reconstructing the transport map via the flow}
\label{sect:flow}
In this section we study the properties of the flow generated by the vector field $\YY_t^N$ defined in \eqref{eq:YN}. As we shall see, we will need to assume that 
$W,V' \in C^r$ with $r \geq 15$.

We consider the flow of  $\YY^N_t$ given by
$$
X_t^N:\mathbb R^N\to \mathbb R^N,\qquad \dot X_t^N=\YY_t^N(X_t^N).
$$
Recalling the form of $\YY_t^N$ (see \eqref{eq:YN}) it is natural to expect that we can give an expansion for $X_t^N$.
More precisely, let us define the flow of $\yy_{0,t}$,
\begin{equation}
\label{eq:X0}
X_{0,t}:\mathbb R\to \mathbb R,\qquad \dot X_{0,t}=\yy_{0,t}(X_{0,t}),\quad X_{0,0}(\lambda)=\lambda,
\end{equation}
and let  $X_{1,t}^N=(X_{1,t}^{N,1},\ldots,X_{1,t}^{N,N}):\mathbb R^N\to \mathbb R^N$ be the solution of the linear ODE 
\begin{equation}
\label{eq:X1}
\begin{split}
\dot X_{1,t}^{N,k}(\lambda_1,\ldots,\lambda_N)&= \yy_{0,t}'(X_{0,t}(\lambda_k))\cdot X_{1,t}^{N,k}(\lambda_1,\ldots,\lambda_N)
+ \yy_{1,t}(X_{0,t}(\lambda_k))\\
&+\int \zz_t(X_{0,t}(\lambda_k), y)\,dM_N^{X_{0,t}}(y)\\
&+\frac1N\sum_{j=1}^N \partial_{2}\zz_{t}\Bigl(X_{0,t}(\lambda_k),X_{0,t}(\lambda_j)\Bigr)\cdot X_{1,t}^{N,j}(\lambda_1,\ldots,\lambda_N)
\end{split}
\end{equation}
with the initial condition $X_{1,t}^N=0$, and $M_N^{X_{0,t}}$ is defined as 
$$
\int f(y) dM_N^{X_{0,t}}(y)=\sum_{i=1}^N \biggl[f(X_{0,t}(\lambda_i))-\int f d\mu_{V_t}\biggr]\qquad \forall\,f \in C_c(\mathbb R).
$$
If we set
$$
X_{0,t}^N(\lambda_1,\ldots,\lambda_N):=\bigl(X_{0,t}(\lambda_1),\ldots,X_{0,t}(\lambda_N)\bigr),
$$
then the following result holds.
\begin{lem} \label{the flow}
Assume that $W,V' \in C^r$ with $r \geq 15$.
Then the flow $X^N_t=(X_{t}^{N,1},\ldots,X_{t}^{N,N}):\mathbb R^N\to \mathbb R^N$ is of class $C^{r-8}$ and the following properties hold:
Let $X_{0,t}$ and $X_{1,t}^N$ be as in \eqref{eq:X0} and \eqref{eq:X1} above,
and define $X^N_{2,t}:\R^N\to\R^N$ via the identity 
$$X_t^N=X_{0,t}^{N}+ \frac1N X_{1,t}^N+\frac{1}{N^2} X_{2,t}^N\,.$$
Then 
\begin{equation}\label{boi1}
\sup_{1\leq k \leq N}\|X_{1,t}^{N,k}\|_{L^4(\mathbb P_{V}^N)} \leq C\log N,\qquad \|X_{2,t}^{N}\|_{L^2(\mathbb P_{V}^N)} \leq CN^{1/2}(\log N)^2,
\end{equation}
where 
$$
\|X_{i,t}^{N}\|_{L^2(\mathbb P_{V}^N)}=\biggl(\int |X_{i,t}^N|^2d\mathbb P_{V}^N\biggr)^{1/2}
,\qquad
|X_{i,t}^N|:=\sqrt{\sum_{j=1,\ldots,N}|X_{i,t}^{N,j}|^2},\quad  i =0,1,2.
$$
In addition, there exists a constant $C>0$ such that, with probability greater than $1- N^{-N/C}$,
\begin{equation}\label{boi2}\max_{1\le k,k'\le N}
|X_{1,t}^{N,k}(\lambda_1,\ldots,\lambda_N)-X_{1,t}^{N,k'}(\lambda_1,\ldots,\lambda_N)|\le C\,\log N \sqrt{N}|\lambda_k-\lambda_{k'}|.\end{equation}

\end{lem}

\begin{proof}
Since $\YY_t^N \in C^{r-8}$ (see Lemma \ref{lem:regular})
it follows by Cauchy-Lipschitz theory that $X_t^N$ is of class $C^{r-8}$.

Using the notation $\hat\lambda=(\lambda_1,\ldots,\lambda_N)\in \mathbb R^N$
and
$$
X_t^{N,k,\sigma}(\hat\lambda):=
X_{0,t}(\lambda_k)+
\sigma\frac{X_{1,t}^{N,k}}{N}(\hat\lambda)+\sigma\frac{X_{2,t}^{N,k}}{N^2}(\hat\lambda)
=(1-\sigma)X_{0,t}(\lambda_k)+\sigma X_t^{N,k}(\hat\lambda)
$$
and defining the measure $M_N^{X_{t}^{N,s}}$ as 
\begin{equation}
\label{eq:MN1}
\int f(y)\, dM_N^{X_{t}^{N,s}}(y)=\sum_{i=1}^N \biggl[f\bigl((1-s)X_{0,t}(\lambda_i)+sX_t^{N,i}(\hat\lambda)\bigr)-\int f \,d\mu_{V_t}\biggr]\qquad \forall\,f \in C_c(\mathbb R).
\end{equation}
by a Taylor expansion we get an ODE for $X_{2,t}^N$:
\begin{equation}
\label{eq:ODE X2 main}
\begin{split}
\dot X_{2,t}^{N,k}(\hat \lambda)&= 
\int_0^1 \yy_{0,t}'\Bigl(X_t^{N,k,s}(\hat\lambda) \Bigr)\,ds\cdot
X_{2,t}^{N,k}(\hat\lambda)\\
&+N\int_0^1\Bigl[\yy_{0,t}'\Bigl(X_t^{N,k,s}(\hat\lambda)\Bigr)
-\yy_{0,t}'\Bigl(X_{0,t}(\lambda_k) \Bigr)\Bigr]\,ds
\cdot X_{1,t}^{N,k}(\hat\lambda)\\
&+ \int_0^1 \yy_{1,t}'\Bigl(X_t^{N,k,s}(\hat\lambda) \Bigr)
\,ds\cdot \Bigl(X_{1,t}^{N,k}(\hat\lambda)+\frac{X_{2,t}^{N,k}(\hat\lambda)}{N} \Bigr)\\
&+\int_0^1 \bigg[
\int \partial_{1}\zz_t\Bigl(X_t^{N,k,s}(\hat\lambda),y \Bigr)\,dM_N^{X_{t}^{N,s}}(y)\\
&\qquad \qquad \qquad \qquad  -
\int \partial_{1}\zz_t\Bigl(X_{0,t}(\lambda_k),y \Bigr)\,dM_N^{X_{0,t}}(y)\bigg]\,ds\cdot \Bigl(X_{1,t}^{N,k}(\hat\lambda)+\frac{X_{2,t}^{N,k}(\hat\lambda)}{N} \Bigr)\\
&+ \int \partial_{1}\zz_{t}\Bigl(X_{0,t}(\lambda_k),y \Bigr)\,dM_N^{X_{0,t}}(y)\cdot \Bigl(X_{1,t}^{N,k}(\hat\lambda)+\frac{X_{2,t}^{N,k}(\hat\lambda)}{N} \Bigr)\\
&+\sum_{j=1}^N\int_0^1 \biggl[\partial_{2}\zz_{t}\Bigl(X_t^{N,k,s}(\hat\lambda),X_t^{N,j,s}(\hat\lambda)\Bigr)-\partial_{2}\zz_{t}\Bigl(X_{0,t}(\lambda_k),X_{0,t}(\lambda_j)\Bigr)\biggr]\,ds\cdot X_{1,t}^{N,j}(\hat\lambda)\\
&+\sum_{j=1}^N\int_0^1 \biggl[\partial_{2}\zz_{t}\Bigl(X_t^{N,k,s}(\hat\lambda),X_t^{N,j,s}(\hat\lambda)\Bigr)\biggr]\,ds\cdot \frac{X_{2,t}^{N,j}(\hat\lambda)}{N},
\end{split}
\end{equation}
with the initial condition $X_{2,0}^{N,k}=0$.
Using that 
$$
\|\yy_{0,t}\|_{C^{r-2}(\R)}\leq C
$$
(see Lemma \ref{lem:regular})
we obtain \begin{equation}
\label{eq:bound X0}
\|X_{0,t}\|_{C^{r-2}(\R)}\leq C.
\end{equation}
We now start to control $X_{1,t}^N$.
First, simply by using that $M_N$ has mass bounded by $2N$ we obtain the rough bound $|X_{1,t}^{N,k}|\leq C\,N$.
Inserting this bound into \eqref{eq:ODE X2 main} one easily obtain the bound
$|X_{2,t}^{N,k}|\leq C\,N^2$.

We now prove finer estimates.
First, by \eqref{conc} together with the fact that $X_{0,t}$ and $x\mapsto \zz_{t}(y, x)$ are Lipschitz (uniformly in $y$), 
it follows that there exists a finite constant $C$ such that, with  probability greater than $1- N^{-N/C}$,
\begin{equation}
\label{eq:psi2MN}
\biggl\| \int \zz_t(\cdot, \lambda)\,dM_N^{X_{0,t}}(\lambda)\biggr\|_\infty\leq C\,
\log N \sqrt{N}.
\end{equation}
Hence it follows easily from \eqref{eq:X1} that
\begin{equation}
\label{eq:bounds X1}
\max_k\|X_{1,t}^{N,k}\|_{\infty}\leq C\log N \sqrt{N}
\end{equation}
outside a set of probability bounded by $N^{-N/C}$.

In order to control $X_{2,t}^N$ we first estimate $X_{1,t}^N$ in $L^4(\mathbb P_{V}^N)$:
using \eqref{eq:X1} again, we get
\begin{multline}
\label{eq:ODE X1}
\frac{d}{dt}\Bigl(\max_{k}\|X_{1,t}^{N,k}\|_{L^4(\mathbb P_{V}^N)}\Bigr)\\
 \leq C\biggl(
\max_{k}\|X_{1,t}^{N,k}\|_{L^4(\mathbb P_{V}^N)}+1
+  \biggl\|\int\zz_t(X_{0,t}(\lambda_k), y)\,dM_N^{X_{0,t}}(y)\biggr\|_{L^4(\mathbb P_{V}^N)}\biggr).
\end{multline}
To bound $X_{1,t}^N$ in $L^4(\mathbb P_{V}^N)$ and then to be able to estimate $X_{2,t}^N$ in 
 $L^2(\mathbb P_{V}^N)$, we will use the following estimates:
\begin{lem} For any $k=1,\ldots,N$,
\begin{equation}
\label{eq:bound psi2 p1}
\biggl\|\int\zz_t(X_{0,t}(\lambda_k), y)\,dM_N^{X_{0,t}}(y)\biggr\|_{L^4(\mathbb P_{V}^N)}
 \leq C \log N,
\end{equation}
\begin{equation}
\label{eq:bound psi2}
\biggl\|\int \partial_{1}\zz_{t}\Bigl(X_{0,t}(\lambda_k),y \Bigr)\,dM_N^{X_{0,t}}(y)\biggr\|_{L^4(\mathbb P_{V}^N)} \leq C\log N.
\end{equation}
\end{lem}\begin{proof}
We write the Fourier decomposition of  $\eta_{t}(x,y):=\zz_t(X_{0,t}(x),X_{0,t}(y))$ to get 
$$\int  \eta_{t}(x,y)\, d M_N(y)=\int \hat  \eta_{t}(x,\xi)  \int e^{i\xi y} \,dM_N(y) \,d\xi\,.$$
Since $\zz_t \in C^{u,v}$ for $u+v\le r-5$
 and $X_{0,t}\in C^{r-2}$ (see \eqref{eq:bound X0}), we deduce that 
$$ |  \hat\eta_{t}(x,\xi) |\le \frac{C}{1+ |\xi|^{r-5}},$$
so that using \eqref{conc3bb} we get 
\begin{align*} 
  \biggl\|\sup_x\biggl|\int \eta_{t}(x,y) \,dM_N(y)\biggr|\biggr\|_{L^4(\mathbb P_{V}^N)}
&\le \int  \Bigl\| \hat\eta_{t}(\cdot,\xi)\Bigr\|_{\infty}  \biggl\| \int e^{i\xi y} dM_N(y)\biggr\|_{L^4(\mathbb P_{V}^N)}\,d\xi \\
&\le C\log N \int \Bigl\|\hat\eta_{t}(\cdot,\xi)\Bigr\|_{\infty} \left(1+|\xi|^7\right)\,d\xi\\
& \le C\log N,
\end{align*}
provided $r>13$.
 The same arguments work
for  $\partial_{1} \zz_{t}$ provided $r>14.$ 
Since by assumption $r \geq 15$, this concludes the proof.
\end{proof}
 Inserting \eqref{eq:bound psi2 p1} into \eqref{eq:ODE X1} we get 
\begin{equation}
\label{eq:X1eps} 
\|X_{1,t}^{N,k}\|_{L^4(\mathbb P_{V}^N)} \leq C\log N \qquad \forall\,k=1,\ldots,N,
 \end{equation}
 which proves the first part of \eqref{boi1}.

We now bound the time derivative of the $L^2$ norm of $X_{2,t}^N$:
using that $M_N$ has mass bounded by $2N$,
in \eqref{eq:ODE X2 main} we can easily estimate
$$
\biggl| N\int_0^1\Bigl[\yy_{0,t}'\Bigl(X_t^{N,k,s}(\hat\lambda)\Bigr)
-\yy_{0,t}'\Bigl(X_{0,t}(\lambda_k) \Bigr)\Bigr]\,ds
\cdot X_{1,t}^{N,k}(\hat\lambda)\biggr|
\leq C|X_{1,t}^{N,k}|^2 + \frac{C}N|X_{1,t}^{N,k}|\,|X_{2,t}^{N,k}|,
$$
\begin{multline*}
\int_0^1 \bigg|
\int \partial_{1}\zz_{t}\Bigl(X_t^{N,k,s}(\hat\lambda),y \Bigr)\,dM_N^{X_{t}^{N,s}}(y)-
\int \partial_{1}\zz_{t}\Bigl(X_{0,t}(\lambda_k),y \Bigr)\,dM_N^{X_{0,t}}(y)\bigg|\,ds\\
\leq C|X_{1,t}^{N,k}| + \frac{C}N|X_{2,t}^{N,k}|+\frac{C}N \sum_j \biggl(|X_{1,t}^{N,j}| + \frac{1}N|X_{2,t}^{N,j}|\biggr),
\end{multline*}
\begin{multline*}
\sum_{j=1}^N\int_0^1 \bigg|\partial_{2}\zz_{t}\Bigl(X_t^{N,k,s}(\hat\lambda),X_t^{N,j,s}(\hat\lambda)\Bigr)
-\partial_{2}\zz_{t}\Bigl(X_{0,t}(\lambda_k),X_{0,t}(\lambda_j)\Bigr)\bigg|\,ds\,|X_{1,t}^{N,j}|\\
\leq \frac{C}N \sum_j \biggl(|X_{1,t}^{N,j}|^2 + \frac{1}N|X_{2,t}^{N,j}|\,|X_{1,t}^{N,j}|\biggr),
\end{multline*}
hence
\begin{align*}
\frac{d}{dt}\|X_{2,t}^N\|_{L^2(\mathbb P_{V}^N)}^2&=2\int \sum_k X_{2,t}^{N,k}\cdot \dot X_{2,t}^{N,k}\,d\mathbb P_{V}^N\\
& \leq C \int \sum_k |X_{2,t}^{N,k}|^2\,d\mathbb P_{V}^N + 
C \int \sum_{k}|X_{1,t}^{N,k}|^2|X_{2,t}^{N,k}|d\mathbb P_{V}^N\\
&+\frac{C}{N} \int \sum_{k}|X_{1,t}^{N,k}| |X_{2,t}^{N,k}|^2d\mathbb P_{V}^N
+C \int \sum_{k}|X_{1,t}^{N,k}| |X_{2,t}^{N,k}|\,d\mathbb P_{V}^N\\
&+\frac{C}{N^2}\int\sum_{k}|X_{2,t}^{N,k}|^3\,d\mathbb P_{V}^N + \frac{C}N\int\sum_{k,j} |X_{1,t}^{N,j}|\,|X_{1,t}^{N,k}|\,|X_{2,t}^{N,k}|\,d\mathbb P_{V}^N\\
&+\frac{C}{N^3}\int\sum_{k,j}|X_{2,t}^{N,k}|^2\,|X_{2,t}^{N,j}|\,d\mathbb P_{V}^N\\
&+\sum_{k}\int X_{2,t}^{N,k}\cdot
\int_0^1 \biggl[\int \partial_{1}\zz_{t}\Bigl(X_{0,t}(\lambda_k),y \Bigr)\,dM_N^{X_{0,t}}(y)\biggr]\,ds\cdot X_{1,t}^{N,k}\,d\mathbb P_{V}^N\\
&+ \frac{C}N\int\sum_{k,j} |X_{1,t}^{N,j}|^2\,|X_{2,t}^{N,k}|\,d\mathbb P_{V}^N+\frac{C}{N^2}\int\sum_{k,j}|X_{2,t}^{N,k}|\,|X_{2,t}^{N,j}|\,|X_{1,t}^{N,j}|\,d\mathbb P_{V}^N\\
&+\frac{C}N \int\sum_{k,j}|X_{2,t}^{N,k}|\,|X_{2,t}^{N,j}|\,d\mathbb P_{V}^N.
\end{align*}
Using the trivial bounds 
$|X_{1,t}^{N,k}|\leq C\,N$ and
$|X_{2,t}^{N,k}|\leq C\,N^2$, \eqref{eq:bound psi2}, 
and elementary inequalities such as, for instance,
$$
\sum_{k,j}|X_{1,t}^{N,j}|\,|X_{1,t}^{N,k}|\,|X_{2,t}^{N,k}|
\leq  \sum_{k,j}\Bigl(|X_{1,t}^{N,j}|^4+ |X_{1,t}^{N,k}|^4 + |X_{2,t}^{N,k}|^2\Bigr),
$$
we obtain
\begin{equation}
\label{eq:ODEX2}
\begin{split}
\frac{d}{dt}\|X_{2,t}^N\|_{L^2(\mathbb P_{V}^N)}^2
& \leq C\bigg( \|X_{2,t}^N\|_{L^2(\mathbb P_{V}^N)}^2 
 + \int \sum_{k}|X_{1,t}^{N,k}|^4\,d\mathbb P_{V}^N\\
&+ \int \sum_{k}|X_{1,t}^{N,k}|^2\,d\mathbb P_{V}^N
+\sum_{k}\log N\,  \|X_{2,t}^{N,k}\|_{L^2(\mathbb P_{V}^N)} \|X_{1,t}^{N,k}\|_{L^4(\mathbb P_{V}^N)}\biggr).
\end{split}
\end{equation}
We now observe, by \eqref{eq:X1eps}, that the last term is bounded by
$$
\|X_{2,t}^N\|_{L^2(\mathbb P_{V}^N)}^2 + (\log N)^2\sum_{k}\|X_{1,t}^{N,k}\|_{L^4(\mathbb P_{V}^N)}^2
\leq \|X_{2,t}^N\|_{L^2(\mathbb P_{V}^N)}^2 + C\,N(\log N)^4.
$$
Hence, using that  $\|X_{1,t}^{N,k}\|_{L^2(\mathbb P_{V}^N)}
\leq \|X_{1,t}^{N,k}\|_{L^4(\mathbb P_{V}^N)}$ and \eqref{eq:X1eps} again,
 the right hand side of \eqref{eq:ODEX2} can be bounded by $C\|X_{2,t}^N\|_{L^2(\mathbb P_{V}^N)} ^2+
C\,N(\log N)^4$, and a Gronwall argument gives
$$
\|X_{2,t}^N\|_{L^2(\mathbb P_{V}^N)}^2 \leq C\, N(\log N)^4,
$$
thus
$$
\|X_{2,t}^N\|_{L^2(\mathbb P_{V}^N)} \leq C\,N^{1/2}(\log N)^2,
$$
concluding the proof of \eqref{boi1}.

We now prove \eqref{boi2}: using \eqref{eq:X1} we have
\begin{align*}
&|\dot X_{1,t}^{N,k}(\hat\lambda)-\dot X_{1,t}^{N,k'}(\hat\lambda)|\\
&\leq
|\yy_{0,t}'(X_{0,t}(\lambda_k))-\yy_{0,t}'(X_{0,t}(\lambda_{k'})) |\,|X_{1,t}^{N,k}(\hat\lambda)|\\
&+|\yy_{0,t}'(X_{0,t}(\lambda_{k'})) |\,|X_{1,t}^{N,k}(\hat\lambda)-X_{1,t}^{N,k'}(\hat\lambda)|
+ |\yy_{1,t}(X_{0,t}(\lambda_k))-\yy_{1,t}(X_{0,t}(\lambda_{k'}))|\\
&+\biggl|\int \Bigl(\zz_t(X_{0,t}(\lambda_k), y)-
\zz_t(X_{0,t}(\lambda_{k'}), y)\Bigr)\,dM_N^{X_{0,t}}(y)\biggr|\\
&+\frac1N\sum_{j=1}^N\int_0^1 \Bigl|\partial_{2}\zz_{t}\Bigl(X_{0,t}(\lambda_k),X_{0,t}(\lambda_j)\Bigr) -\partial_{2}\zz_{t}\Bigl(X_{0,t}(\lambda_{k'}),X_{0,t}(\lambda_j)\Bigr) \Bigr|\,ds\,|X_{1,t}^{N,j}(\hat\lambda)|.
\end{align*}
Using that $|X_{0,t}(\lambda_k)-X_{0,t}(\lambda_{k'})| \leq C|\lambda_k - \lambda_{k'}|$,
the bound \eqref{eq:bounds X1},
the Lipschitz regularity of $\yy_{0,t}'$, $\yy_{1,t}$, $\zz_t$,
and $\partial_{2}\zz_{t}$, and
the fact that
$$
\biggl\| \int \partial_{1}\zz_{t}(\cdot, \lambda)\,dM_N^{X_{0,t}}(\lambda)\biggr\|_\infty\leq C\,
\log N \sqrt{N}
$$
with probability greater than $1- N^{-N/C}$ (see \eqref{conc}),
we get
$$
|\dot X_{1,t}^{N,k}(\hat\lambda)-\dot X_{1,t}^{N,k'}(\hat\lambda)|
\leq C|X_{1,t}^{N,k}(\hat\lambda)- X_{1,t}^{N,k'}(\hat\lambda)|+C\log N \sqrt{N}|\lambda_k - \lambda_{k'}|
$$
outside a set of probability less than $ N^{-N/C}$,
so \eqref{boi2} follows from Gronwall.

\end{proof}

\section{Transport and universality}\label{sect:univ}

In this section we prove 
Theorem \ref{thm:univ} on
universality using the regularity properties of the approximate transport
maps obtained in the previous sections.

\begin{proof}[Proof of Theorem \ref{thm:univ}]
Let us first remark that the map $T_0$ from Theorem \ref{thm:T} 
coincides with $X_{0,1}$, where $X_{0,t}$ is the flow defined in \eqref{eq:X0}.
Also, notice that $X^N_1:\R^N\to\R^N$
is an approximate transport of $\mathbb P^N_V$ onto $\mathbb P_{V+W}^N$
(see Lemma \ref{lemma:flow} and Proposition \ref{prop:key}).
Set $\hat X_1^N:=X_{0,1}^N+\frac{1}{N}X_{1,1}^N$, with $X_{0,t}^N$ and $X_{1,t}^N$ as in Lemma \ref{the flow}.
Since $X_1^N-\hat X_1^N=\frac{1}{N^2}X_{2,1}^N$,
recalling \eqref{boi1} and using H\"older inequality to control the $L^1$ norm with the $L^2$ norm, we see that
\begin{equation}
\label{eq:Xhat}
\begin{split}
\biggl|\int g(\hat X_1^N)\,d\mathbb P_V^N - \int g(X_1^N)\,d\mathbb P_V^N  \biggr|&
\leq \|\nabla g\|_\infty \frac{1}{N^2} \int |X_{2,1}^N|\,d\mathbb P_V^N \\
&\leq \|\nabla g\|_\infty \frac{1}{N^2} \|X_{2,1}^N\|_{L^2(\mathbb P_V)} \\
&\leq C \|\nabla g\|_\infty \frac{(\log N)^2}{N^{3/2}}.
\end{split}
\end{equation}
This implies that also $\hat X^N_1:\R^N\to\R^N$
is an approximate transport of $\mathbb P^N_V$ onto $\mathbb P_{V+W}^N$.
In addition, we see that $\hat X_1^N$ preserves the order of the $\lambda_i$ with large probability.
Indeed, first of all $X_{0,t}:\R\to \R$ is the flow of $\yy_{0,t}$ which is Lipschitz with some constant $L$, hence
differentiating \eqref{eq:X0} we get
$$
\frac{d}{dt}\bigl|X_{0,t}'\bigr| \leq \bigl|\yy_{0,t}'(X_{0,t})\bigr|\,\bigl|X_{0,t}'\bigr| \leq L\,\bigl|X_{0,t}'\bigr|,
\qquad X_{0,0}'=1,
$$
so Gronwall's inequality gives the bound
$$
e^{-Lt}\leq \bigl|X_{0,t}'\bigr|\leq e^{Lt}
$$
Since $X_{0,0}'=1$, it follows by continuity that
$X_{0,t}'$ must remain positive for all time and it satisfies
\begin{equation}
\label{eq:x0prime}
e^{-Lt}\leq X_{0,t}'\leq e^{Lt},
\end{equation}
from which we deduce that
$$
e^{-Lt}\bigl(\lambda_j - \lambda _i\bigr) \leq X_{0,t}(\lambda_j) - X_{0,t}(\lambda _i)
\leq e^{Lt}\bigl(\lambda_j - \lambda _i\bigr), \qquad \forall\,\lambda_i<\lambda_j.
$$
In particular,
$$
e^{-L}\bigl(\lambda_j - \lambda _i\bigr) \leq X_{0,1}(\lambda_j) - X_{0,1}(\lambda _i)
\leq e^{L}\bigl(\lambda_j - \lambda _i\bigr).
$$
Hence, using the notation $\hat\lambda=(\lambda_1,\ldots,\lambda_N)$,
since
$$
\biggl|\frac{1}{N}X_{1,t}^{N,j}(\hat\lambda)-\frac{1}NX_{1,t}^{N,i}(\hat\lambda)\biggr|\le C\,\frac{\log N}{ \sqrt{N}}|\lambda_i-\lambda_{j}|
$$
(see \eqref{boi2}) with probability greater than $1-N^{-N/C}$
we get
$$
\frac{1}{C}\bigl(\lambda_j - \lambda _i\bigr) \leq \hat X_{1}^{N,j}(\hat \lambda) - \hat X_{1}^{N,i}(\hat\lambda )
\leq C\bigl(\lambda_j - \lambda _i\bigr)
$$
with probability greater than $1-N^{-N/C}$.

We now make the following observation: the ordered measures $\tilde P_V^N$
and $\tilde P_{V+W}^N$ are obtained as the image of $\mathbb P_V^N$ and $\mathbb P_{V+W}^N$
via the map
$R:\R^N \to \R^N$ defined as
$$
[R(x_1,\ldots,x_N)]_i:=\min_{\sharp J=i}\max_{j\in J} x_j.
$$
Notice that this map is $1$-Lipschitz for the sup norm.

Hence, if $g$ is a function of $m$-variables
we have
$
\|\nabla(g\circ R)\|_\infty \leq \sqrt{m} \|\nabla g\|_\infty$, so by Lemma \ref{lemma:flow}, Proposition \ref{prop:key}, and \eqref{eq:Xhat}, we get
$$
\biggl|\int g\circ R (\hat X_1^N) \,d\mathbb P_V^N-\int g\circ R\, d\mathbb P_{V+W}^N\biggr|
\leq C\frac{(\log N)^3 }{N}\|g\|_\infty+C\sqrt{m}\,\frac{(\log N)^2}{N^{3/2}}  \|\nabla g\|_\infty.
$$
Since $\hat X_1^N$ preserves the order with probability greater than $1-N^{-N/C}$,
we can replace $g\circ R (N \hat X_1^N)$ with $g(N \hat X_1^N\circ R)$ up to a very small error bounded by $\|g\|_\infty N^{-N/C}$. Hence, since $R_\# \mathbb P_V^N=\tilde P_V^N$
and $R_\# \mathbb P_{V+W}^N=\tilde P_{V+W}^N$, we deduce that,
for any Lipschitz function $f:\R^m\to \R$, 
\begin{multline*}
\bigg|\int f\bigl(N(\lambda_{i+1}-\lambda_i),\ldots,N(\lambda_{i+m}-\lambda_i)\bigr)\, d\tilde P_{V+W}^N\qquad\\
\quad -\int f\Bigl(N\bigl(\hat X_1^{N,i+1}(\hat\lambda)-\hat X_1^{N,i}(\hat\lambda)\bigr),
\ldots,N\bigl(\hat X_1^{N,i+m}(\hat\lambda)-\hat X_1^{N,i}(\hat\lambda)\bigr)\Bigr)
d\tilde P_{V}^N\bigg| \\
\le C\, \frac{(\log N)^3 }{N}\|f\|_\infty+C\,\sqrt{m}\,\frac{(\log N)^2}{N^{1/2}}\|\nabla f\|_\infty.
\end{multline*}
Recalling that 
$$
\hat X_1^{N,j}(\hat\lambda)=X_{0,1}(\lambda_j)+\frac{1}{N}X_{1,1}^{N,j}(\hat\lambda),
$$
we observe that, 
as $X_{0,1}$ is of class $C^2$,
$$X_{0,1}(\lambda_{i+k})-X_{0,1}(\lambda_{i})=X_{0,1}'(\lambda_i)\,(\lambda_{i+k}-\lambda_{i})+O(|\lambda_{i+k}-\lambda_{i}|^2).$$
Also, by \eqref{boi2} we deduce that,
out of a set of probability bounded by $ N^{-N/C}$,
\begin{equation}\label{toto}
|X_{1,1}^{N,i+k}(\hat\lambda)-X_{1,1}^{N,i}(\hat\lambda)|\le  C\,\log N \sqrt{N}|\lambda_{i+k}-\lambda_i|\,.
\end{equation}
As $X_{0,1}'(\lambda_i) \geq e^{-L}$ (see \eqref{eq:x0prime}) we deduce
$$\frac{1}{N}\,|X_{1,1}^{N,i+k}(\hat\lambda)-X_{1,1}^{N,i}(\hat\lambda)|\le C\,|X_{0,1}'(\lambda_i)\,(\lambda_{i+k}-\lambda_{i})| \,\frac{\log N}{{N}^{1/2}}$$
and
$$
O(|\lambda_{i+k}-\lambda_{i}|^2)=O\bigl(\bigl|X_{0,1}'(\lambda_i)\,(\lambda_{i+k}-\lambda_{i})\bigr|^2\bigr)
$$
hence with probability greater than $1- N^{-N/C}$ it holds
\begin{align*}
\hat X_1^{N,i+k}(\hat\lambda)-\hat X_1^{N,i}(\hat\lambda)=
X_{0,t}'(\lambda_i)\,(\lambda_{i+k}-\lambda_{i}) \biggl[1+
O\biggl(\frac{\,\log N}{N^{1/2}} \biggr)+O\bigl(\bigl|X_{0,1}'(\lambda_i)\,(\lambda_{i+k}-\lambda_{i})\bigr|\bigr)\biggr]\,.
\end{align*}
Since we assume $f$  supported in $[-M,M]^m$, the domain of integration is restricted to $\hat\lambda$
such that $\{NX_{0,t}'(\lambda_i)\,(\lambda_i-\lambda_{i+k})\}_{1\le k \le m}$  is bounded by $2M$ for $N$ large enough, therefore
$$\hat X_1^{N,i+k}(\hat\lambda)-\hat X_1^{N,i}(\hat\lambda)=
X_{0,t}'(\lambda_i)\,(\lambda_{i+k}-\lambda_{i}) +O\biggl(2M\frac{\,\log N}{N^{3/2}}\biggr)
+O\biggl(\frac{4M^2}{N^2}\biggr),$$
from which the first bound follows easily.

For the second point
we observe that $a_{V+W}=X_{0,1}(a_V)$ and, arguing as before,
\begin{multline*}
\bigg|
\int f\bigl(N^{2/3}(\lambda_1-a_{V+W}), \ldots,N^{2/3}(\lambda_m-a_{V+W})\bigr) \,d\tilde P_{V+W}^N\\
-\int f\Bigl(N^{2/3}\bigl(\hat X_1^{N,1}(\hat\lambda)-X_{0,1}(a_V)\bigr), \ldots,N^{2/3}\bigl(\hat X_1^{N,m}(\hat\lambda)-X_{0,1}(a_V)\bigr)\Bigr)\, d\tilde P_{V}^N\bigg|\\
\le C\, \frac{(\log N)^3 }{N}\,\|f\|_\infty+C\,\sqrt{m}\frac{(\log N)^2}{N^{5/6}}\,\|\nabla f\|_\infty.
\end{multline*}
Since, by \eqref{boi1},
\begin{align*}
\hat X_1^{N,i}(\lambda)&=X_{0,1}(\lambda_i) +O_{L^4(\mathbb P_{V}^N)}\biggl(\frac{\log N}{N}\biggr)\\
&= X_{0,1}(a_V) +X_{0,1}'(a_V)\,(\lambda_i-a_V)+O\bigl(|\lambda_i-a_V|^2\bigr)+O_{L^4(\mathbb P_{V}^N)}\biggl(\frac{\log N}{N}\biggr),
\end{align*}
we conclude as in the first point.
\end{proof}

\bibliographystyle{alpha}
\bibliography{trander}

\end{document}